\documentclass[leqno,11pt]{amsart}
\usepackage{amssymb, amsmath,amsmath,latexsym,amssymb,amsfonts,amsbsy, amsthm,esint}
\usepackage{color}
\usepackage{graphicx}
\usepackage{tikz}
\usetikzlibrary{arrows, calc, decorations.markings,intersections, patterns,patterns.meta}
\usepackage{caption}
\usepackage{subcaption}
\usepackage{hyperref}

\usepackage{pgfplots}
\pgfplotsset{compat=1.15}
\usepackage{mathrsfs}

\makeatletter
\DeclareFontFamily{U}{tipa}{}
\DeclareFontShape{U}{tipa}{m}{n}{<->tipa10}{}
\newcommand{\arc@char}{{\usefont{U}{tipa}{m}{n}\symbol{62}}}%

\newcommand{\arc}[1]{\mathpalette\arc@arc{#1}}

\newcommand{\arc@arc}[2]{%
  \sbox0{$\m@th#1#2$}%
  \vbox{
    \hbox{\resizebox{\wd0}{\height}{\arc@char}}
    \nointerlineskip
    \box0
  }%
}
\makeatother

\let\pa=\partial
\let\al=\alpha

\let\g=\gamma
\let\d=\delta
\let\z=\zeta
\let\lam=\lambda

\let\s=\sigma
\let\t=\theta
\let\f=\frac

\let\G= \Gamma
\let\D=\Delta

\let\Om=\Omega

\let\e=\varepsilon
\let\pa=\partial

\let\ri=\rightarrow
\let\na=\nabla

\def\non{\nonumber}

\newcommand{\beq}{\begin{equation}}
\newcommand{\eeq}{\end{equation}}
\newcommand{\beqo}{\begin{equation*}}
\newcommand{\eeqo}{\end{equation*}}
\newcommand{\ben}{\begin{eqnarray}}
\newcommand{\een}{\end{eqnarray}}
\newcommand{\beno}{\begin{eqnarray*}}
\newcommand{\eeno}{\end{eqnarray*}}

\newtheorem{theorem}{Theorem}[section]
\newtheorem{definition}[theorem]{Definition}
\newtheorem{lemma}[theorem]{Lemma}
\newtheorem{proposition}[theorem]{Proposition}

\newtheorem{Theorem}{Theorem}[section]

\newtheorem{Remark}[Theorem]{Remark}

\theoremstyle{remark}
\newtheorem{step}{Step}
\newtheorem{case}{Case}
\newtheorem{rmk}{Remark}[section]

\newcommand{\dist}{\mathrm{dist}}

\newcommand{\BR}{\mathbb{R}}

\newcommand{\cl}{\mathcal{L}^1}

\newcommand{\hg}{\hat{\G}^1_{\e,\g}}

\begin{document}
\title[Triple junction problem on the plane]{On the triple junction problem on the plane without symmetry hypotheses}

\author{Nicholas D. Alikakos}
\address{Department of Mathematics, University of Athens (EKPA), Panepistemiopolis, 15784 Athens, Greece}
\email{nalikako@math.uoa.gr}

\author{Zhiyuan Geng}
\address{Basque Center for Applied Mathematics, Alameda de Mazarredo 14
48009 Bilbao, Bizkaia, Spain}
\email{zgeng@bcamath.org}

\date{\today}

\begin{abstract}
    We investigate the Allen-Cahn system
    \begin{equation*}
    \Delta u-W_u(u)=0,\quad u:\mathbb{R}^2\rightarrow\mathbb{R}^2,
    \end{equation*}
    where $W\in C^2(\mathbb{R}^2,[0,+\infty))$ is a potential with three global minima. We establish the existence of an entire solution $u$ which possesses a triple junction structure. The main strategy is to study the global minimizer $u_\varepsilon$ of the variational problem 
    \begin{equation*}
    \min\int_{B_1} \left( \f{\varepsilon}{2}|\nabla u|^2+\f{1}{\varepsilon}W(u) \right)\,dz,\ \ u=g_\varepsilon \text{ on }\partial B_1.
    \end{equation*}
    The point of departure is an energy lower bound that plays a crucial role in estimating the location and size of the diffuse interface. We do not impose any symmetry hypothesis on the solution. 
\end{abstract}


\maketitle

\section{Introduction}
This paper is concerned with the existence of an entire, bounded, minimizing solution to the system 
\beq\label{main equation}
\D u-W_u(u)=0, \quad u:\BR^2\ri\BR^2,
\eeq
where $W$ is a triple-well potential with three global minima. A key feature of the present work is that we make no a priori hypotheses of symmetry on the solution. Specifically for $W$ we assume
\vspace{3mm}
\begin{enumerate}
    \item[(H1).]  $W\in C^2(\BR^2;[0,+\infty))$, $\{z: \,W(z)=0\}=\{a_1,a_2,a_3\}$, $W_u(u)\cdot u >0$ if $|u|>M$ and 
    \beqo
     c_2|\xi^2| \geq \xi^TW_{uu}(a_i)\xi\geq  c_1|\xi|^2,\; i=1,2,3.
    \eeqo 
    for some positive constants $c_1<c_2$ depending on $W$. 
\item[(H2).] For $i\neq j$, $i,j\in \{1,2,3\}$, let $U_{ij}\in W^{1,2}(\BR,\BR^2)$ be an 1D minimizer of the \emph{action}
\beqo
\sigma_{ij}:=\min \int_{\BR}\left(\f12|U_{ij}'|^2+W(U_{ij})\right)\,d\eta, \quad \lim\limits_{\eta\ri-\infty}U_{ij}(\eta)=a_i,\ \lim\limits_{\eta\ri+\infty}U_{ij}(\eta)=a_j.
\eeqo
$\sigma_{ij}$ satisfies
\beq\label{cond on sigma}
\sigma_{ij}\equiv \sigma>0\quad \text{ for }i\neq j\in\{1,2,3\} \text{ and some constant }\sigma.
\eeq
\end{enumerate}

Note that \eqref{main equation} is the Euler-Lagrange equation corresponding  to 
\begin{equation}\label{energy functional}
J(u,\Omega):=\int_\Om \left( \f12|\na u|^2+W(u) \right)\,dz,\ \ \forall \text{ bounded open set } \Om\subset \BR^2. 
\end{equation}

It is an easy fact that if $J(u,\BR^2)<\infty$ for a solution $u$ of \eqref{main equation}, then $u$ is a constant (cf. \cite{alikakos2011some}). For this reason the construction of the solution cannot be achieved directly. Moreover, the definition of the \emph{minimizing solution} is as follows.
\begin{definition}
$u$ is a \emph{minimizing solution} of \eqref{main equation} in the sense of De Giorgi if 
\begin{equation}
    J(u,\Om)\leq J(u+v,\Om), \quad \forall \text{ bounded open set } \Om\subset \BR^2, \ \forall  v\in C_0^1(\Om).
\end{equation}
\end{definition}

We note that the hypothesis (H2) on $W$ that all $\sigma_{ij}$ are equal is mainly for convenience and does not imply any symmetry for the potential or the solution. In particular, one can construct examples of non-symmetric potentials (i.e. $W$ is not in the equivariant class of the symmetry group of the equilateral triangle) with three connections that have equal actions $\sigma_{ij}\equiv \sigma$. Take $W(z)=\vert z^3-1 \vert^2$, $z\in\mathbb{C}$. The connections $U_{ij}$ are explicitly known (see \cite[Page 79]{afs-book}) and they avoid a neighborhood $B_\rho$ of $z=0$. When $\rho$ is sufficiently small one can modify $W$ in $B_\rho$ in a non-symmetric way and keep the connections unaffected at the same time. 

For the scalar solution $u:\BR^n\ri \BR$ of the \eqref{main equation}, there is a relationship (cf. \cite{pacard2012role, wei2012geometrization, chodosh2019lecture}) between minimizing solutions and minimal surfaces. Many deep results have been obtained in the process of understanding this relationship, see \cite{ambrosio2000entire,farina20111d,ghoussoub1998conjecture,savin2009regularity,wang2017new,del2011giorgi,del2013entire,guaraco2018min,liu2017global,pacard2013stable} and the references therein. In the vector case, minimizing solutions are related to minimal partitions, see Baldo\cite{Baldo}, Sternberg\cite{sternberg1988effect} and Fonseca \& Tartar \cite{fonseca1989gradient}.

The solution we are after can be considered as the diffuse analog of the singular minimal cone on the plane (unique up to translations and rotations), the so-called \emph{triod}, that provides a minimal partition $\mathcal{P}=\{D_1,D_2,D_3\}$ of the plane into three $120^\circ$ degree sectors. Minimality is related to Steiner's classical result which states that given three points $A,B,C$ on the plane such that the corresponding triangle has no angle greater or equal to $120^\circ$, then if $P$ is a point that minimizes the sum of the distances $|P-A|+|P-B|+|P-C|$, the line segments $PA,PB,PC$ form three $120^\circ$ degree angles. 

We now state our main results. 

\begin{theorem}\label{main theorem}
Fix $\g< \min\{\g_0,\min\limits_{i,j\in\{1,2,3\}} \f12 \vert a_i-a_j\vert, \sqrt{\f{\sigma}{20C_W}}\}$, where $\g_0, C_W$ are constants defined later. There is a constant $C_0$ which depends on $\g,W$ such that under the hypotheses (H1), (H2),  there exists an entire, bounded minimizing solution of \eqref{main equation} with the following triple junction structure. 
\begin{enumerate}
     \item[a.] For every $r>0$, there exists a point $P(r)$ such that $\vert u(P(r))-a_1\vert \leq \g$.
    \item[b.] There exists $\{Q_j\}_{j=1}^\infty\cup \{R_j\}_{j=1}^\infty$ such that 
    \beqo
    \begin{split}
    &\vert u(Q_j)-a_2\vert\leq \g,\quad \vert u(R_j)-a_3\vert\leq \g,\\
    &\dist(Q_j, 0)\leq 32jC_0,\quad \dist(R_j, 0)\leq 32jC_0.
    \end{split}
    \eeqo
    The pairwise distance of any two points from $\{Q_j\}_{j=1}^\infty\cup \{R_j\}_{j=1}^\infty$ is larger than or equal to $6C_0$.
    \item[c.] For each $Q_j$, there exists $P_j$ such that 
    \beqo
    \dist(Q_j,P_j)\leq C_0,\quad \vert u(P_j)-a_1\vert\leq \g;
    \eeqo
    This property also holds for $\{R_j\}$.
    \item[d.]  For any sequence $r_k\ri+\infty$ one can extract a subsequence, still denoted by $\{r_k\}$,  such that 
\beq\label{convergence in main thm}
u(r_k x)\rightarrow u_0(x)\ \text{ in }L^1_{loc}(\BR^2),
\eeq
where $u_0(x)=\sum\limits_{i=1}^3a_i \mathcal{\chi}_{D_i}$. Here $\chi$ is the characteristic function. $\mathcal{P}=\{D_1,D_2,D_3\}$ provides a minimal partition of $\BR^2$ into three sectors of the angle $\f23\pi$ and $\pa \mathcal{P}$ is a triod centered at $0$. The partition $\mathcal{P}$ may depend on $\{r_k\}$. Also there exists a sequence $\{r'_k\}_{k=1}^\infty$, such that for each $\xi\in D_j\, (j=1,2,3)$, $\vert \xi \vert=1$, 
\beq
\lim\limits_{k\ri\infty}\f{1}{r_k'}\int_0^{r_k'} u(s\xi)\,ds=a_j,   \tag{\ref*{convergence in main thm}'} \label{convergence in main thm along a rays}
\eeq
and the convergence is uniform for $\xi$ in compact sets of $\mathbb{S}^1\setminus \pa\mathcal{P}$.
\end{enumerate}
\end{theorem}

In 1996, Bronsard, Gui and Schatzman \cite{bronsard1996three} established the existence of an entire triple junction solution for the triple-well potential $W$ in the equivariant class of the reflection group $\mathcal{G}$ of the symmetries of the equilateral triangle. Their solution satisfies the estimate $|u(x)-a_i|\leq Ke^{-kd(x,\pa D_i)},\ i=1,2,3$ and hence in particular it connects the minima of $W$ at infinity along rays emanating from the origin, $u(\lam \f{x}{|x|})\ri a_i$ as $\lam\ri+\infty$, for $x\in D_i$, $i=1,2,3$. And thus it is a stronger result on the asymptotic behavior of the solution than our theorem above. In spite of its great importance, this result suffers from the fact that their solution is obtained as a minimizer in the equivariant class $u(gx)=gu(x), g\in \mathcal{G}$, hence is not necessarily even stable under general perturbations. 

Since that time, the problem has been understood for an abstract reflection group on $\BR^n$ via different and quite general methods rendering also the complete stratification of the solution. The general theory under optimal hypotheses on $W$ covers in particular the diffuse analogs of the two miminal singular cones in $\BR^3$: the triod with a spine and the tetrahedral Plateau complex (\cite{taylor1976structure}). The second one was originally obtained in 2008 by Gui and Schatzman \cite{gui2008symmetric}. We refer to the book \cite{afs-book} and to the references therein. 

In 2021, Fusco \cite{fusco} succeeded in establishing essentially the result of \cite{bronsard1996three} in the equivariant class of the rotation subgroup of $\mathcal{G}$ (by $\f23\pi$), thus eliminating the two reflections. This was a significant achievement since the problem now cannot be reduced to a fundamental domain containing a single minimum of $W$. Nevertheless it suffers from the same shortcoming as before, not allowing general perturbations. Symmetry fixes the center of the junction at the origin, which is a fact that simplifies considerably the analysis. 

Schatzman \cite{schatzman2002asymmetric} is probably the only previously known entire minimizing solution that does not assume any symmetry. Her result later was revisited in \cite{fusco2017layered,monteil2017metric,smyrnelis2020connecting}. Triple junction solutions for bounded domains without symmetry assumptions were constructed by Sternberg and Ziemer \cite{sternberg1994local} for clover-shaped domains in $\BR^2$ via $\Gamma$-convergence, and for more general domains by Flores, Padilla and Tonegawa \cite{flores2001higher} by the means of a mountain pass argument. These results do not seem to provide estimates that allow to pass to the limit and establish the existence of a triple junction solution on $\BR^2$. 

We now list some of the key steps in the proof of our results. We begin by rescaling the problem on the unit disk. 
\begin{equation}\label{rescaling functional}
\min\limits_{u=g_\e\text{ on }\pa B_1} \int_{B_1(0)} \left( \f{\e}{2}|\na u|^2+\f{1}{\e}W(u) \right)\,dz=:\min\limits_{u=g_\e\text{ on }\pa B_1} J_\e(u),\quad z=(x,y)\in \BR^2,
\end{equation}
where $g_\e$ is a given smooth function connecting the phases ($a_1\ri a_2$, $a_2\ri a_3$, $a_3\ri a_1$) in $O(\e)$-intervals $I_i$, i=1,2,3, and otherwise equals to constants ($a_1,a_2,a_3$) in the complement $\pa B_1\setminus \cup_{i=1}^3 I_i$. A possible choice of $g_{\e}$ is given in \eqref{def of g_eps}.

First a general remark is in order. The main thrust of our work is in obtaining estimates for the minimizer(s) of \eqref{rescaling functional}, that is working at the $\e$--level. We make use of the limiting minimal partition problem and $\G$--convergence only in the final stages of the proof of Property (d) in Theorem \ref{main theorem}, specifically \eqref{convergence in main thm} and \eqref{convergence in main thm along a rays}.

The first step is the derivation of a tight upper/lower bound for the minimizer, 
\begin{equation}\label{intro:lower bound}
    3\s-C\e^{\f12}\leq \int_{B_1}\left( \f{\e}{2}|\na u_\e|^2+\f{1}{\e}W(u_\e) \right)\,dz\leq 3\s+C\e.
\end{equation}

The usefulness of such an estimate for extracting qualitative information for the minimizer as $\e\ri0$ was the major point in the joint paper \cite{AF} by G. Fusco and one of the authors. However the geometry of the examples treated in \cite{AF} is simple, and the transition along the interface involves only one of the variables, hence only part of the gradient in \eqref{intro:lower bound} was involved. In the present work the transition is 2D and so the whole gradient is present. Thus the proof of \eqref{intro:lower bound} utilizes several new ideas, and in particular it introduces the center $(x^*,y^*)$ of the junction. We remark that the particular power $\f12$ in the lower bound above does not have a particular significance in our treatment. 

The second step is the derivation of the estimate
\begin{equation}
\label{intro:y* est} |x^*|,\, |y^*|\leq C\e^{\f14},    
\end{equation}
which makes use of both the lower and the upper bound in \eqref{intro:lower bound}.

The third step is the localization of the \emph{diffuse interface}
\beq
\label{intro:def of diffuse interface} I_{\e,\g}:=\{z\in B_1: |u_\e(z)-a_i|>\g\text{ for all }i=1,2,3\}
\eeq
in an $O(\e^{\f14})$ neighborhood of the triod. The proof of this utilizes all the previous estimates and also the variational maximum principle in \cite{AF2}  (see also \cite[Theorem 4.1]{afs-book}).

The final tool utilized in the proof is an estimate on the width of the diffuse interface $I_{\e,\g}$. Let 
\beq
    \label{intro: Gi} \G_{\e,\g}^i:=\{z\in \overline{B}_1: |u_{\e}(z)-a_i|=\g\},\ \ i=1,2,3. 
\eeq
For an arbitrary $z_1\in \G^i_{\e,\g}$ set 
\beq
\label{intro: def r1} r_1:=\max\{r: B(z_1,r)\cap (\bigcup\limits_{j\neq i}\G_{\e,\g}^j)=\varnothing\}.
\eeq
Then we have the estimate
\beq\label{intro: est r1}
r_1\leq C(\g,W)\e. 
\eeq

This estimate provides a substitute for the lack of an $O(\e)$ localization of the diffuse interface from the triod, which may very well be an effect of the absence of symmetry. The derivation of \eqref{intro: est r1} is mainly based on the vector version of the Caffarelli-C\'{o}rdoba density estimate \cite{AF3} (see also \cite[Theorem 5.2]{afs-book}). Specifically, it does not utilize any of the previous estimates.

Next utilizing these estimates in the proof of the key Lemma \ref{New ver of lemma 6.2}, for any $k\in\mathbb{N}^+$ and $\g$ small, when $\e<\e(k,\g,W)$ we can obtain a set of points $\{Q_1^\e,Q_2^\e,...,Q_{2^k}^\e,R_1^\e,R_2^\e,...,R_{2^k}^\e\}$ in the ball $B_{\f12}(0)$ satisfying
\begin{align}
&\dist(Q_j^\e,\G^1_{\e,\g})\leq C_0\e,\ \dist(R_j^\e,\G^1_{\e,\g})\leq C_0\e,\quad \forall j=1,...,2^k.\\
\label{intro:QR}&
\vert u_\e(Q_j^\e)-a_2\vert\leq \g,\ \vert u_\e(R_j^\e)-a_3\vert\leq \g,\ \ \forall j=1,...,2^k.\\
\label{intro:distance btw Qi Ri}&
\dist(A,B)\geq 6C_0\e,\quad  \forall A,B\in\{Q_1^\e,...,Q_{2^k}^\e,R_1^\e,...,R_{2^k}^\e\}.
\end{align}
Here $C_0=C(\g,W)$ is the constant in \eqref{intro: est r1}. This argument has a more discrete nature and is accomplished via an elaborate induction argument. With Lemma \ref{New ver of lemma 6.2}, we can study the blow up map at the scale $\e$ and let $\e\ri 0$ to obtain a minimizing solution $u(\cdot)$ of \eqref{main equation}, depending on $\g$, which satisfies Properties (a), (b), (c) in Theorem \ref{main theorem}.

Finally we come to the derivation of Property (d). We begin with the minimizing solution $u:\BR^2\ri\BR^2$ satisfying Properties (a), (b), (c) in Theorem \ref{main theorem}.  Modica \cite{modica1979gamma} proved that entire minimizers $u:\BR^n\ri\BR$ of the scalar Allen--Cahn equation, with bistable $W$, converge along subsequences to a minimizing cone $u_0$,
\beqo
u(r_kx)\ri u_0(x)\ \text{ in }L^1_{loc}, \ r_k\ri+\infty.
\eeqo
His proof utilizes the monotonicity formula for minimal surfaces (see Giusti \cite[Theorem 9.3]{giusti1984minimal}). Modica's argument can be adjusted to the problem at hand, where minimal partitions are the relevant objects. A convenient set-up is that of flat chains introduced in Fleming \cite{fleming1966flat} and further developed by White \cite{whitenotes}. In particular minimal partitions can be identified with minimal flat chains of top dimension, which also satisfy the monotonicity formula. Consequently we obtain that $u_0$ in \eqref{convergence in main thm} is a planar minimizing cone, and therefore either a straight line or a triod. However, Properties (a), (b), (c)  together imply an energy lower bound that excludes the straight line, and so \eqref{convergence in main thm} is established. \eqref{convergence in main thm along a rays} follows from the cone property of $u_0$ via \cite[Proposition 5.6]{afs-book}. This concludes the sketch of the proof of Theorem \ref{main theorem}.

Concluding, we remark that all the arguments in the present paper utilize the variational character of $u$ only, with the exception of the pointwise exponential estimate in \eqref{distance to a_i} that makes use of linear elliptic theory.

A previous version of this work was uploaded in the arXiv in December, 2022. In a personal communication, Peter Sternberg has brought to our attention that in a joint work with Etienne Sandier, they have obtained comparable results.

The article is organized as follows. In Section \ref{sec:preliminary} we present some preliminary results from \cite{AF} and \cite{afs-book}. In Section \ref{sec:lower bound} we establish the lower bound in \eqref{intro:lower bound}. Then we show in Section \ref{sec:localization} the estimate \eqref{intro:y* est} and the localization of the diffuse interface within an $O(\e^{\f14})$ neighborhood of the triod. Estimate \eqref{intro: est r1} of the width of the diffuse interface is established in Section \ref{sec:width}. Finally in Section \ref{sec:proof of main theorem} we give the proof of Theorem \ref{main theorem}. In the Appendix, we give a proof of the upper bound in \eqref{upper bound}. A similar estimate was derived in Fusco \cite{fusco}. 

\section*{Acknowledgement} 

We would like to thank Arghir Zarnescu for his interest in this work, for the stimulating discussions and for his great hospitality. 

Z. Geng is supported by the Basque Government through the BERC 2022-2025 program and by the Spanish State Research Agency through BCAM Severo Ochoa excellence accreditation SEV-2017-0718 and through project PID2020-114189RB-I00 funded by Agencia Estatal de Investigaci\'{o}n (PID2020-114189RB-I00/AEI/10.13039/501100011033).

\section{Preliminaries} \label{sec:preliminary}

Throughout the paper we denote by $z=(x,y)$ a 2D point and by $B(z,r)$ the 2D ball centered at the point $z$ with radius $r$. In addition, we let $B_1$ denote the unit 2D ball centered at the origin. We recall the following basic results which are important in our analysis.
\begin{lemma}[Lemma 2.1 in \cite{AF}]\label{lemma: potential energy estimate}
The hypotheses on $W$ imply the existence of $\delta_W>0$, and constants $c_W,C_W>0$ such that
\beqo
\begin{split}
&|u-a_i|=\delta\\
\Rightarrow & \ \f12 c_W\delta^2\leq W(u)\leq \f12 C_W \delta^2,\quad \forall \delta<\delta_W,\ i=1,2,3.
\end{split}
\eeqo
Moreover if $\min\limits_{i=1,2,3} |u-a_i|\geq \delta$ for some $\delta<\delta_W$, then $W(u)\geq \f12 c_W\delta^2.$
\end{lemma}

\begin{lemma}[Lemma 2.3 in \cite{AF}]\label{lemma: 1D energy estimate}
Take $i\neq j \in \{1,2,3\}$, $\d <\d_W$ and $s_+>s_+$ be two real numbers. Let $v:(s_-,s_+)\ri \BR^2$ be a smooth map that minimizes the energy functional 
\beqo
J_{(s_-,s_+)}(v):=\int_{s_-}^{s_+} \left(\f12|\na v|^2+W(v)\right)\,dx 
\eeqo
subject to the boundary condition 
\beqo
|v(s_-)-a_i|=|v(s_+)-a_j|=\delta.
\eeqo
Then
\beqo
J_{(s_-,s_+)}(v)\geq \sigma-C_W\delta^2,
\eeqo
where $C_W$ is the constant in Lemma \ref{lemma: potential energy estimate}. 
\end{lemma}

For most of the paper, we consider the variational problem 
\beq\label{epsilon energy}
\min\int_{B_1}\left( \f{\e}{2}|\na u|^2+\f1\e W(u)\right)\,dz,
\eeq
where $W(u)$ satisfies Hypotheses (H1) and (H2). We denote by $u_\e \in W^{1,2}(B_1,\BR^2)$ a global minimizer of the functional \eqref{epsilon energy} with respect to the boundary condition in polar coordinates 
$$
u_\e(1,\theta)=g_\e(\t)\text{ on }\pa B_1,
$$
where $g_\e:[0,2\pi)$ is given by
\begin{equation}\label{def of g_eps}
g_\e(\theta):=\begin{cases}
a_2+g_0(\f{\t-\f{\pi}{2}+c_0\e}{2c_0\e})(a_1-a_2), &  \t\in [\f{\pi}{2}-c_0\e, \f{\pi}{2}+c_0\e)\\
a_1,& \t\in[\f{\pi}{2}+c_0\e, \f76\pi-c_0\e),\\
a_1+g_0(\f{\t-\f{7\pi}{6}+c_0\e}{2c_0\e})(a_3-a_1), &  \t\in [\f76\pi-c_0\e, \f76\pi+c_0\e)\\
a_3,& \t\in[\f76\pi+c_0\e, \f{11}{6}\pi-c_0\e),\\
a_3+g_0(\f{\t-\f{11\pi}{6}+c_0\e}{2c_0\e})(a_2-a_3), &  \t\in [\f{11}{6}\pi-c_0\e, \f{11}{6}\pi+c_0\e),\\
a_2,& \t\in[\f{11}{6}\pi+c_0\e, 2\pi)\cap[0,\f{\pi}{2}-c_0\e),
\end{cases}
\end{equation}
where $g_0: [0,1]\ri [0,1]$ is a strictly increasing smooth function that satisfies $g_0(0)=0$, $g_0(1)=1$ and $|g_0'(x)|\leq 2$. From the definition there exists a positive constant $M$ such that $|g_\e|\leq M$. 

The energy $J_\e(u_\e)$ satisfies the following upper bound.
\begin{lemma}\label{lemma:upper bound}
There is a constant $C=C(W)$ such that  
\begin{equation}\label{upper bound}
\int_{B_1}\left\{  \f{\e}{2}|\na u_\e|^2+\f{1}{\e}W(u_\e) \right\}\,dz\leq 3\sigma +C\e.
\end{equation}
\end{lemma}

A similar upper bound is derived by Fusco in \cite[estimate (3.18)]{fusco}. For completeness we present the proof of Lemma \ref{lemma:upper bound} in Appendix \ref{app:upper bound}.

In addition, we can control $|u_\e|$ and $|\na u_\e|$ thanks to the smoothness assumption of $W$ and standard elliptic regularity theory. 

\begin{lemma}
Let $u_\e$ minimize the functional \eqref{epsilon energy} with the boundary condition $u_\e=g_\e$ on $\pa B_1$. There is a constant $M$ independent of $\e$, such that  
\begin{equation}\label{gradient bound}
|u_\e(z)|\leq M,\quad |\na u_\e(z)|\leq\f{M}{\e}, \quad \forall z\in B_1. 
\end{equation}
\end{lemma}

We omit the proof.

\section{Lower bound for $J_\e(u_\e)$} \label{sec:lower bound}

\begin{proposition}\label{lower bound epsilon 1/3}(weak lower bound)
There exist constants $C_1$ and $\e_0$, such that for any $\e\leq \e_0$, it holds 
\beq\label{lower bound e 1/3}
\int_{B_1} \left(\f{\e}{2}|\na u_\e|^2+\f{1}{\e}W(u_\e)\right)\,dz\geq 3\sigma-C_1\e^{\f13}.
\eeq 

\end{proposition}

\begin{proof}
For the sake of convenience we write $u_\e=u$ throughout the proof. Firstly we set the family of horizontal line segments $\g_y$ for $y\in [-\f12,1]$ as
\beqo
\g_y:=\{(x,y):\; x\in\mathbb{R}\}\cap B_1.
\eeqo


Then we define functions $\lam_1(y)$, $\lam_2(y)$, $\lam_3(y)$ for $y\in [-\f12,1]$,
\beqo
\lam_i(y):=\mathcal{L}^1(\g_y\cap \{|u(x,y)-a_i|<\e^{\f16}\}),\quad i\in \{1,2,3\}.
\eeqo
Here $\mathcal{L}^1$ denotes the 1-dimensional Lebesgue measure. Then by the boundary condition we know for any $y\in [-\f12+c_0\e, 1-c_0\e)$, it holds that $\lam_1(y)>0$ and $\lam_2(y)>0$. 

Let $y^*$ be the constant defined by
\beqo
y^*:= \min\{y\in [-\f12+c_0\e,1]: \; \lam_1(y)+\lam_2(y)\geq \mathcal{L}^1(\g_y)-\e^{\f13}\}.
\eeqo


Denote the subsets
\beqo
\begin{split}
&\Om_1:= B_1\cap \{(x,y): y\geq y^*\},\\
&\Om_2:=B_1\cap  \{(x,y): y< y^*\}.
\end{split} 
\eeqo

We will calculate the energy in $\Om_1$ and $\Om_2$ respectively. In $\Om_1$, if $y^*\geq 1-c_0\e$, we simply estimate $\int_{\Om_1}\e|\na u|^2+\f{1}{\e}W(u)\,dx\,dy\geq 0$. Otherwise for any $y^*<y<1-c_0\e$, \eqref{def of g_eps} implies that 
\beqo
u(-\sqrt{1-y^2},y)=a_1,\quad u(\sqrt{1-y^2},y)=a_2.
\eeqo
Integrating the energy density on $\g_y$ gives
\beqo
\int_{-\sqrt{1-y^2}}^{\sqrt{1-y^2}} \left(\f\e2|\pa_{x} u|^2+\f{1}{\e}W(u)\right) \ dx\geq \sigma.
\eeqo
Therefore we obtain
\beq\label{energy on Om_1} 
\begin{split}
&\int_{\Om_1}\left(\f\e2|\na u|^2+\f{1}{\e}W(u)\right)\ dx\,dy\\
\geq &\int_{y^*}^{1-c_0\e}\int_{-\sqrt{1-y^2}}^{\sqrt{1-y^2}}\left(\f\e2|\pa_{x} u|^2+\f{1}{\e}W(u)\right)\ dx\,dy\geq \max\{\sigma(1-c_0\e-y^*), 0\}.
\end{split}
\eeq

On domain $\Om_2$, we claim that there exists a constant $C$ such that 
\begin{equation*}
    \mathcal{L}^1(\{y: -\f12+c_0\e<y<y^*,\; \lam_3(y)=0\})<C\e^{\f13}.
\end{equation*}

To prove this claim, first we note that when $y^*<-\f12+c_0\e+C\e^{\f13}$ (for some constant $C$ to be chosen later) the statement is trivial. Therefore we only consider the situation $y^*\geq -\f12+c_0\e+C\e^{\f13}$. Set 
\beqo
S:=\{y: -\f12+c_0\e<y<y^*,\; \lam_3(y)=0\}.
\eeqo
For any $y\in S$, definitions of $y^*$ and $S$ imply that $\lam_1(y)+\lam_2(y)+\lam_3(y)<\mathcal{L}^1(\g_y)-\e^{\f13}$, i.e.
\begin{equation*}
    \cl(\{x \in [-\sqrt{1-y^2},\sqrt{1-y^2}]: |u(x,y)-a_i|>\e^{\f16},\ \forall i\})>\e^{\f13}.
\end{equation*}
By our assumption on the potential function $W$, 
\beqo
W(u(x,y))\geq \f12c_W\e^{\f13}, \quad \text{when } |u(x,y)-a_i|>\e^{\f16},\ \forall i. 
\eeqo
From the energy upper bound \eqref{upper bound} we get
\beqo
\begin{split}
    &4\sigma\geq \f{1}{\e}\int_S\int_{\g_y} W(u)\ dxdy\geq  \f{c_W}{2\e}\cl(S)\e^{\f13}\e^{\f13}\\
    &\Rightarrow \cl(S)\leq C\e^{\f13}\quad \text{for some constant }C\text{ depending on }W. 
\end{split}
\eeqo

Now we calculate the energy in $\Om_2$. If $-\f12\leq y^*<-\f12+c_0\e+C\e^{\f13}$, then we define the set
\beqo
K_0:= \{ x\in [-\f{\sqrt{3}}{2}+c_0\e, \f{\sqrt{3}}{2}-c_0\e]: |u(x,y^*)-a_i|<\e^{\f16},\ i=1 \text{ or }2\}.
\eeqo
According to the definition of $y^*$, $\cl(K_0)\geq \sqrt{3}-2c_0\e-\e^{\f13}$. We have
\beq\label{energy on Om_2 first case}
\begin{split}
    &\int_{\Om_2} \left(\f\e2|\na u|^2+\f{1}{\e}W(u)\right)\ dx\,dy\\
    \geq & \int_{K_0}\int_{-\sqrt{1-x^2}}^{y^*} \left(\f\e2|\pa_{y} u|^2+\f{1}{\e}W(u)\right) \ dy \,dx\geq (\sqrt{3}-2c_0\e-\e^{\f13})(\sigma-C_W\e^{\f13}),
 \end{split}
\eeq
where $C_W$ is a constant only depending on the potential $W$ and the last estimate follows from Lemma \ref{lemma: 1D energy estimate}.  \eqref{energy on Om_1} and \eqref{energy on Om_2 first case} imply that when  $-\f12\leq y^*<-\f12+c_0\e+C\e^{\f13}$,
\begin{equation}\label{energy estimate first case}
\int_{B_1} \left(\f\e2|\na u|^2+\f{1}{\e}W(u)\right)\ dx\,dy\geq (\f32+\sqrt{3}-3c_0\e-(C+1)\e^{\f13})(\s-C_W\e^{\f13}),
\end{equation}
which satisfies \eqref{lower bound e 1/3} when $\e$ is sufficiently small. 

Now it suffices to consider the case $y^*\geq -\f12+c_0\e+C\e^{\f13}$. For any $x\in [-\f{\sqrt{3}}{2}+c_0\e, \f{\sqrt{3}}{2}-c_0\e]$, we set
\beqo
\zeta(x):=\min\{y^*, \sqrt{1-x^2}\}.
\eeqo
We also introduce the sets $K,\ M$. 
\begin{align*}
&K:= \{ x\in [-\f{\sqrt{3}}{2}+c_0\e, \f{\sqrt{3}}{2}-c_0\e]: \ |u(x,\zeta(x))-a_i|<\e^{\f16},\ i=1\text{ or }2\},\\
& M:=\{y\in [-\f12+c_0\e, y^*]:\ \lam_3(y)>0 \}.
\end{align*}

Moreover, let $\theta\in(0,\f{\pi}{2})$ be a parameter that will be determined later. In the next step we will split the potential $W$ into two parts $W=(\sin^2\theta) W+(\cos^2\theta) W$ and compute the energy in the vertical direction and the horizontal direction respectively.  

For any $x\in K$,
\begin{equation*}
    \ |u(x,\zeta(x))-a_i|<\e^{\f16}, \text{ for }i=1 \text{ or 2},\ u(x,-\sqrt{1-x^2})=a_3.
\end{equation*}
We estimate the energy in the vertical direction,
\begin{equation}\label{Om2 vertical}
\begin{split}
    &\int_{-\sqrt{1-x^2}}^{\zeta(x)} \left(\f\e2|\pa_{y}u|^2+\f{\sin^2\theta}{\e}W(u)\right)\ dy\\
    =& \sin\theta \int_{-\sqrt{1-x^2}}^{\zeta(x)} \left(\f\e{2\sin\theta}|\pa_{y}u|^2+\f{\sin\theta}{\e}W(u)\right)\ dy\\
    \geq & \sin\theta(\sigma-C_W\e^{\f13}).
    \end{split}
\end{equation}

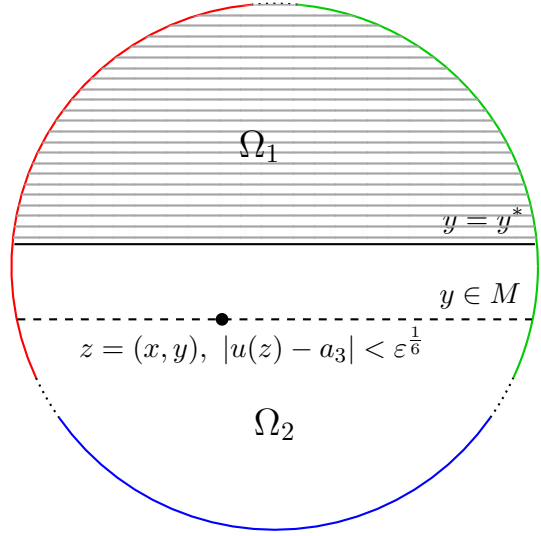
\begin{figure}[htt]
\begin{tikzpicture}[thick]
\draw [black!20!green,domain=0:85] plot ({3.5*cos(\x)}, {3.5*sin(\x)});
\draw [black!20!green,domain=335:360] plot ({3.5*cos(\x)}, {3.5*sin(\x)});
\draw [dotted,domain=85:95] plot ({3.5*cos(\x)}, {3.5*sin(\x)});
\draw [red,domain=95:205] plot ({3.5*cos(\x)}, {3.5*sin(\x)});
\draw [dotted,domain=205:215] plot ({3.5*cos(\x)}, {3.5*sin(\x)});
\draw [blue,domain=215:325] plot ({3.5*cos(\x)}, {3.5*sin(\x)}) node at (0,-2.1) [color=black] {\Large{$\Om_2$}};
\draw [dotted,domain=325:335] plot ({3.5*cos(\x)}, {3.5*sin(\x)});
\draw (-3.46,0.3)--(3.46,0.3) node[above left] {$y=y^*$};
\draw [dashed](-3.42,-0.7)--(3.42,-0.7) node[above left] {$y\in M$};
\filldraw(-0.7,-0.7) circle (2pt) node at (-0.3,-1.05){$z=(x,y),\ |u(z)-a_3|<\e^{\f16}$}; 
\path ({-3.5*cos(5)},{3.5*sin(5)}) coordinate (A);
\path ({3.5*cos(5)},{3.5*sin(5)}) coordinate (B);
\fill[pattern={Lines[angle=-45,distance=4pt]}, pattern color=gray!70] (A)--(B) arc[start angle=5, end angle=175, radius=3.5] node at (-0.2,1.6) {\Large{$\Om_1$}};
\end{tikzpicture}
\caption{red curve: $a_1$, green curve: $a_2$, blue curve: $a_3$, shaded region: $\Om_1$, non-shaded region: $\Om_2$.}
\label{pic1}
\end{figure}

Similarly, for any $y\in M$, it holds that
 \beqo
 u(-\sqrt{1-y^2},y)=a_1,\ u(\sqrt{1-y^2},y)=a_2,\ \exists (x_0,y)\in \g_y \text{ s.t. }|u(x_0,y)-a_3|<\e^{\f16}.
 \eeqo
 \begin{equation}\label{Om2 horizontal}
     \begin{split}
    &\int_{\g_y}  \left(  \f\e2|\pa_{x}u|^2+\f{\cos^2\theta}{\e}W(u) \right)\ dx\\
    =& \cos\theta \left\{ \int_{-\sqrt{1-y^2}}^{x_0} +\int_{x_0}^{\sqrt{1-y^2}}\right\} \left(\f{\e}{2\cos\theta}|\pa_{x}u|^2+ \f{\cos\theta}{\e} W(u)\right)\ dx\\
    \geq & 2\cos\theta(\s-C_W\e^{\f13}).
     \end{split}
 \end{equation}
 
Using \eqref{Om2 vertical} and \eqref{Om2 horizontal} we obtain
\begin{equation*}
             \int_{\Om_2} \left( \f\e2|\na u|^2+\f{1}{\e}W(u) \right)\,dx\,dy
        \geq  \left(\sin\theta\cl(K)+2\cos\theta\cl(M)\right)(\s-C_W\e^{\f13}).
\end{equation*}

Since the estimate above holds for any $\theta\in(0,\f{\pi}{2})$, we can maximize the lower bound by taking $\theta=\arctan\left(\f{\cl(K)}{2\cl(M)}\right)$. Consequently we have
\begin{equation}
\label{compute in Om2}
\int_{\Om_2} \left( \f\e2|\na u|^2+\f{1}{\e}W(u) \right)\,dx\,dy\geq \sqrt{\cl(K)^2+4\cl(M)^2} (\s-C_W\e^{\f13}).    
\end{equation}

Moreover, by the definition of $K,M$ and the boundary condition \eqref{def of g_eps},
\beqo
{\cl}(K)\geq \sqrt{3}-2c_0\e-\e^{\f13},\quad {\cl}(M)\geq y^*+\f12-C\e^{\f13}-c_0\e.
\eeqo

As a consequence, there exists a constant $C$ such that for small enough $\e$, 
\beqo
\sqrt{\cl(K)^2+4\cl(M)^2}\geq \sqrt{3+4(y^*+\f12)^2}-C\e^{\f13}.
\eeqo
From this and \eqref{compute in Om2} we obtain
\begin{equation}\label{energy in Om2}
    \int_{\Om_2} \left( \f\e2|\na u|^2+\f{1}{\e}W(u) \right)\,dx\,dy\geq \sqrt{3+4(y^*+\f12)^2}\cdot \s -C\e^{\f13},
\end{equation}
where $C$ is a constant independent of $\e$. Finally, we combine the estimates \eqref{energy on Om_1} and \eqref{energy in Om2} to get
\beqo
\begin{split}
  &\int_{B_1} \left( \f\e2|\na u|^2+\f{1}{\e}W(u) \right)\,dx\,dy\\
\geq  & (1-y^*+\sqrt{3+4(y^*+\f12)^2})\sigma -C\e^{\f13}-c_0\s\e\\
\geq & 3\sigma-C\e^{\f13},
\end{split}
\eeqo
which is just the lower bound \eqref{lower bound e 1/3}. Note that in the last step we have used the fact that the function $1-x+\sqrt{3+4(x+\f12)^2}$ obtains its minimal value $3$ at $x=0$.   

\end{proof}

We can slightly modify the proof above to get the following refinement of the lower bound estimate. 

\begin{proposition}[lower bound of order $\e^{\f12}$] \label{prop epsilon 1/2}
There exist constants $C(W)$ and $\e_1$, such that for any $\e\leq \e_1$, it holds 
\beq\label{lower bound e 1/2}
\int_{B_1} \left(\f{\e}{2}|\na u_\e|^2+\f{1}{\e}W(u_\e)\right)\,dz\geq 3\sigma-C\e^{\f12}.
\eeq 
\end{proposition}

\begin{proof}
Define all the analogous functions and subsets corresponding to $\e^{\f12}$. 
\begin{align}
\nonumber    &\lam_i:= \cl(\g_y\cap\{|u(x,y)-a_i|<\e^{\f14}\}),\ \ i\in\{1,2,3\},\\
\label{def of y*}    &y^*:= \min\{y\in [-\f12,1]: \lam_1(y)+\lam_2(y)\geq \cl(\g_y)-\al\e^{\f12}\},\\ 
\nonumber    &\Om_1:= B_1\cap \{(x,y): y\geq y^*\},\quad \Om_2:=B_1\cap  \{(x,y): y< y^*\},\\
\nonumber    &\zeta(x):=\min\{y^*,\sqrt{1-x^2}\},\\
\nonumber    & K:= \{ x\in [-\f{\sqrt{3}}{2}+c_0\e, \f{\sqrt{3}}{2}-c_0\e]: \ |u(x,\zeta(x))-a_i|<\e^{\f14},\ i=1\text{ or }2\},\\
\nonumber    & M:=\{y\in [-\f12+c_0\e, y^*]:\ \lam_3(y)>0 \}.
\end{align}
Here $\al$ is a constant only depending on the potential function $W$, whose value will be determined later. Also we emphasize that in the rest of the paper we use \eqref{def of y*} as the definition of $y^*$.

When $-\f12\leq y^*<-\f12+c_0\e$, using the same calculation as in \eqref{energy on Om_2 first case} and \eqref{energy estimate first case} we know the energy is strictly larger than $3\sigma$ when $\e$ is suitably small. Thus it suffices to consider the case $y^*\geq -\f12+c_0\e$, for which the set $M$ is well defined. 

Note that for $y\in [-\f12+c_0\e,y^*]\setminus M$, it holds that
\begin{align}
    \nonumber &\lam_1(y)+\lam_2(y)+\lam_3(y)<\cl(\g_y)-\al \e^{\f12},\\
    \label{energy estimate outside M} &\quad \int_{\g_y} W(u)\,dx\geq \f12c_W\al \e. 
\end{align}

We split $W=\f34W+\f14W$. Thanks to the analogous estimates as in \eqref{Om2 vertical}and \eqref{Om2 horizontal}, we have
\beq\label{energy in Om2, e^1/2}
\begin{split}
    &\int_{\Om_2}\left( \f\e2|\na u|^2+\f{1}{\e}W(u) \right)\,dx\,dy\\
    \geq &\int_K \int_{-\sqrt{1-x^2}}^{\zeta(x)} \left(\f\e2|\pa_{y}u|^2+\f{3}{4\e}W(u)\right)\,dy\, dx+ \int_M \int_{\g_{y}}  \left(  \f\e2|\pa_{x}u|^2+\f{1}{4\e} W(u) \right)\,dx\,dy\\
    &\qquad + \int_{[-\f12+c_0\e,y^*]\setminus M}\int_{\g_{y}} \f{1}{4\e} W(u)\,dx\,dy\\
    \geq &\left(\f{\sqrt3}{2} \cl(K)+\cl(M)\right)(\sigma-C_W\e^{\f12})+ \beta \f{c_W\al}{8},
\end{split}
\eeq
where $\beta:= y^*+\f12-c_0\e-\cl(M)\geq 0$. We have  
\begin{equation}\label{energy on B1, e^1/2}
    \begin{split}
        &\int_{B_1}\left( \f\e2|\na u|^2+\f{1}{\e}W(u) \right)\,dx\,dy\\
        = &\int_{\Om_2}\left( \f\e2|\na u|^2+\f{1}{\e}W(u) \right)\,dx\,dy+\int_{\Om_1}\left( \f\e2|\na u|^2+\f{1}{\e}W(u) \right)\,dx\,dy\\
        \geq & \left(\f{\sqrt3}{2} \cl(K)+\cl(M)\right)(\sigma-C_W\e^{\f12})+  \f{c_W\al\beta}{8}+\s(1-c_0\e-y^*)\\
        \geq & \left(\f{\sqrt3}{2} (\sqrt3-2c_0\e-\al\e^{\f12})+\cl(M)\right)(\sigma-C_W\e^{\f12})+\\
        &\quad + \f{c_W\al\beta}{8}+\s(1-c_0\e-y^*)\\
        = & 3\sigma -(2+y^*)C_W\e^{\f12}+\beta(\f{c_W\al}{8}-\sigma+C_W\e^{\f12})\\
        &\quad - \f{\sqrt{3}}{2} \al\e^{\f12}(\s-C_W\e^{\f12})-c_0\e \left[ (\s-C_W\e^{\f12})(1+\sqrt{3})+\s \right]\\
        \geq & 3\s-(3C_W+\f{\sqrt{3}}{2}\al\s)\e^{\f12}+\beta(\f{c_W\al}{8}-\sigma+C_W\e^{\f12})-(2+\sqrt3)c_0\s\e+\f{\sqrt{3}}{2}\al C_W\e.
    \end{split}
\end{equation}
Note that from the second line to the third line we have used the estimates \eqref{energy in Om2, e^1/2} and \eqref{energy on Om_1}.
Now we choose $\al=\al(W)$ such that $\f{c_W\al}{8}\geq 2\s$. Then it follows easily from \eqref{energy on B1, e^1/2} that there exists a constant $C=C(W)$ such that the lower bound \eqref{lower bound e 1/2} holds for sufficiently small $\e$. 
\end{proof}

\section{Localization of the transition layer}\label{sec:localization}

Let $\lam_i\ (i=1,2,3)$, $y^*$, $\Om_j\ (j=1,2)$, $\zeta(x)$, $K$, $M$ be defined in the same way as in the proof of Proposition \ref{prop epsilon 1/2}. Also recall the definition $\beta:=y^*+\f12-c_0\e-\cl(M)$. We can further derive the following lemma.

\begin{lemma}
\label{estimate of y*}
There exists a constant $C$, which depends on the potential functional $W$, such that 
\begin{equation}\label{ineq est y^*}
    |y^*|\leq C\e^{1/4},
\end{equation}
where $y^*$ is defined in \eqref{def of y*}.
\end{lemma}
\begin{proof}
First of all, by the energy upper bound \eqref{upper bound} and $\f{c_W\al}{8}\geq 2\sigma$, we get from \eqref{energy on B1, e^1/2}
\beq\label{estimate of beta}
\beta \leq C\e^{\f12},
\eeq
where $C$ is a constant depending only on $W$. Now we take the ratio $\kappa=\f{2L^1(M)}{L^1(K)}$ and calculate in the same way as in the proof of Proposition \ref{lower bound epsilon 1/3} to get
\begin{equation}\label{est of y*: lower bound}
   \int_{B_1}\left( \f{\e}{2}|\na u_\e|^2+\f{1}{\e}W(u_\e) \right)\,dz \geq (1-y^*+\sqrt{3+4(y^*+\f12)^2})\sigma -C\e^{\f12}.
\end{equation}
Combining \eqref{est of y*: lower bound} and \eqref{upper bound} yields
\beqo
\begin{split}
&1-y^*+\sqrt{3+4(y^*+\f12)^2}\leq 3+C\e^{\f12}\\
\Rightarrow &\quad  |y^*|\leq C\e^{\f14}\ \ \text{ for some constant }C=C(W).    
\end{split}
\eeqo
\end{proof}

\begin{Remark}
    By applying the procedure above to each of the ``legs" of the triod we obtain an equilateral triangle centered at the origin and of size O($\e^{\f14}$). The ``center" of the junction $(x*,y*)$ is located in this triangle. 
\end{Remark}

Let $T$ denote the union of three line segments (the so-called \emph{triod}), each two of which form an angle of $\f{2\pi}{3}$.
\beqo
T:=\left\{(0,y),\ y\in [0,1]\right\}\cap \left\{(x,\f{\sqrt{3}}{3}x),\ x\in[-\f{\sqrt{3}}{2},0]\right\}\cap \left\{(x,-\f{\sqrt{3}}{3}x),\ x\in [0,\f{\sqrt{3}}{2}]\right\}.
\eeqo

$T$ divides the closed unit disk $\overline{B}_1$ into three regions:
\beqo
D_i:=\{(r\cos\theta,r\sin\theta): \ 0<r<1,\ \f{2(i-2)\pi}{3}<\theta<\f{2(i-1)\pi}{3}\},\quad i=1,2,3.
\eeqo

We mention in passing that according to the $\G$-convergence result in Gazoulis\cite{gazoulis}, it holds that
\beqo
\lim\limits_{\e\ri 0} \|u_\e-u_0\|_{L^1(B_1)}=0,
\eeqo
where
\beqo
u_0=\sum_{i=1}^3 a_i\chi_{D_i}.
\eeqo
The proposition below provides a quantitative refinement of this convergence. 

For $\g\geq 0$, we recall the diffuse interface is defined as 
\beqo
I_{\e,\g}:=\{z\in B_1:\  |u_\e(z)-a_i|>\g \text{ for all }i=1,2,3\}.
\eeqo

The main result of this section is to show that $I_{\e,\g}$ is contained in a $O(\e^{\f14})$ neighborhood of $T$, for fixed $\g$ and suitably small $\e$. We have the following proposition.

\begin{proposition}\label{location of diffused interface}
There exists a constant $\g_0$ such that for any $0<\g\leq \g_0$, there exist constants $C=C(\g,W)$ and $\e(\g,W)$ satisfying
\beq\label{diffused interface close to Gamma}
 \forall \e<\e(\g,W),\quad I_{\e,\g}\subset N_{C\e^{\f14}}(T):=\{z\in B_1: \dist(z,T)\leq C\e^{\f14}\}.
\eeq
Moreover, there are positive constants $K$ and $k$ such that
\beq\label{distance to a_i}
|u_\e(z)-a_i|\leq Ke^{-\f{k}{\e}(\dist(z,\G)-C\e^{\f14})^+},\quad z\in D_i,\ i=1,2,3.
\eeq
\end{proposition}

\begin{figure}[ht]
\begin{tikzpicture}[thick]
\draw [black!20!green,domain=0:85] plot ({3.5*cos(\x)}, {3.5*sin(\x)});
\draw [black!20!green,domain=335:360] plot ({3.5*cos(\x)}, {3.5*sin(\x)});
\draw [dotted,domain=85:95] plot ({3.5*cos(\x)}, {3.5*sin(\x)});
\draw [red,domain=95:205] plot ({3.5*cos(\x)}, {3.5*sin(\x)});
\draw [dotted,domain=205:215] plot ({3.5*cos(\x)}, {3.5*sin(\x)});
\draw [blue,domain=215:325] plot ({3.5*cos(\x)}, {3.5*sin(\x)});
\draw [dotted,domain=325:335] plot ({3.5*cos(\x)}, {3.5*sin(\x)});
\draw (-3.48,0.2)--(3.48,0.2) node[above left] {$y=y^*$};
\draw [dashed] ({3.5*cos(210)},{3.5*sin(210)})--(0,0);
\draw[dashed] ({3.5*cos(150)},{3.5*sin(150)})--({3.5*cos(330)},{3.5*sin(330)});
\draw [dashed] (0,3.5)--(0,0);
\path (-0.6, {sqrt(3.5*3.5-0.6*0.6)}) coordinate (B);
\draw (-0.6,{0.6/sqrt(3)})--(B) node at (-0.7,2.5) [right] {$l_1^t$};
\path[name path=line] (-0.6,{0.6/sqrt(3)})--(-3.6,{0.6/sqrt(3)-sqrt(3)});
\path[name path=circle] (0,0) circle (3.5);
\path[name intersections={of= line and circle, by=A}];
\draw (A)--(-0.6,{0.6/sqrt(3)}) node at (-2.3,-0.75) [right]{$l_2^t$};
\fill[pattern={Lines[angle=-45,distance=7pt]},pattern color=red] (A)--(-0.6,{0.6/sqrt(3)}) --(B) arc (100:200:3.5);
\draw [<->] (-0.6,1.2)--(0,1.2) node at (-0.3,1.2) [below] {$t$}; 
\end{tikzpicture}
\caption{Defintion of $l_1^t$, $l_2^t$. For any $z$ in the red shaded region, $|u(z)-a_1|\leq \gamma$.}
\label{pic2}
\end{figure}
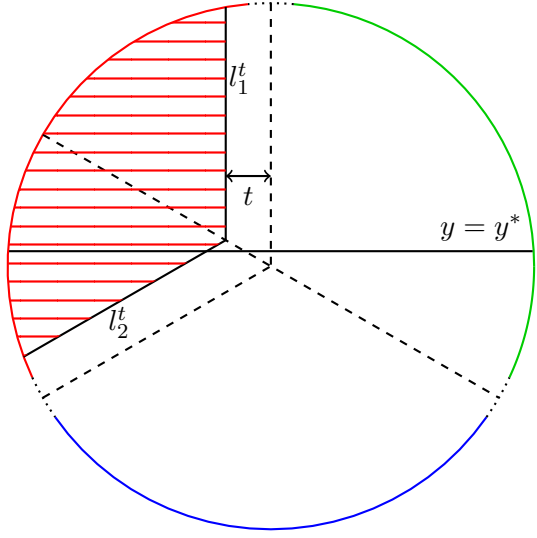

\begin{proof}
For simplicity we write $u_\e=u$. We fix $C_0$ as the constant in \eqref{ineq est y^*}, i.e. 
\beq\label{ineq est y* 2}
y^*\leq C_0\e^{\f14}.
\eeq

For any $t\in [0,\f12]$, we define the line segments (see Figure \ref{pic2})
\begin{equation}
\begin{split}
    &l_1^t:=\{(-t,y):\ y\geq \f{\sqrt{3}}{3}t\}\cap \bar{D}_1,\\
    &l_2^t:=\{(x,y):\ \f{y-\f{\sqrt{3}}{3}t}{x+t}=\f{\sqrt3}{3},\; x\leq -t\}\cap \bar{D}_1. 
\end{split}
\end{equation}
Thanks to \eqref{ineq est y* 2}, we have
\beqo
t\geq \sqrt3 C_0 \e^{\f14}\; \Rightarrow  \; l_1^t\subset \overline{\Om}_1. 
\eeqo

Let $\g_0$ be a positive constant which will be determined later. Now we fix $\g\leq \g_0$ and set
\begin{equation*}
    A:=\{t: t\in [\sqrt3C_0\e^{\f14},\f12], \; \max\limits_{z\in l_1^t} |u(z)-a_1|>\g\}.
\end{equation*}
If $A=\varnothing$ then we can proceed to \eqref{def of B} below. We will show the measure of $A$ is of order $\e^{\f14}$. The proof will utilize that the part of the lower bound over $\Om_1$ derived above is based entirely on the horizontal gradient, and thus the vertical gradient can be added to produce an improvement. See \cite[Lemma 4.3]{AF} for a similar idea. 

From the definition we know that for any $t\in A$, $l_1^t\subset \overline{\Om}_1$ and there exists a point $z_t\in l_1^t$ such that 
\beqo
|u(z_t)-a_1|>\g.
\eeqo
Also from the boundary condition \eqref{def of g_eps} it follows that
\beqo
u((-t,\sqrt{1-t^2}))=a_1, \quad (-t,\sqrt{1-t^2})\in  l_1^t\cap \pa B_1, \quad t\geq c_0.
\eeqo
Therefore for any $t\in A$, there exists a constant $C_1:=C_1(\g,W)$ such that 
\beq\label{energy on l_1^t}
\int_{\f{\sqrt3t}{3}}^{\sqrt{1-t^2}} \f{\e}{2}|\pa_y u|^2+\f{1}{\e}W(u)\,dy\geq C_1.
\eeq
Set 
\beqo
\kappa:=\f{\mathcal{H}^1(A)C_1}{\s(1-c_0\e-y^*)},
\eeqo
we recalculate the energy on $\Om_1$ using \eqref{energy on Om_1} and \eqref{energy on l_1^t}:
\begin{equation}\label{est on Om1: new}
    \begin{split}
        &\int_{\Om_1} \left(\f{\e}{2}|\na u|^2+\f{1}{\e}W(u)\right)\,dxdy\\
        = & \int_{\Om_1} \left( \f{\e}{2}|\pa_x u|^2+\f{1}{(1+\kappa^2)\e}W(u) \right)\,dxdy+\int_{\Om_1}\left( \f{\e}{2}|\pa_y u|^2+\f{\kappa^2}{(1+\kappa^2)\e} W(u) \right)\,dydx\\
        \geq & \f{1}{\sqrt{1+\kappa^2}} \int_{\Om_1} \left(\f{\e\sqrt{1+\kappa^2}}{2} |\pa_x u|^2 +\f{1}{\sqrt{1+\kappa^2}\e}W(u)  \right)\,dxdy\\
        &\qquad  + \f{\kappa}{\sqrt{1+\kappa^2}} \int_{\Om_1}\left(\f{\sqrt{1+\kappa^2}\e}{2\kappa} |\pa_y u|^2+\f{\kappa}{\sqrt{1+\kappa^2}\e} W(u)  \right)\,dydx\\
        \geq & \f{1}{\sqrt{1+\kappa^2}}\sigma(1-c_0\e-y^*)+\f{\kappa}{\sqrt{1+\kappa^2}}\int_A\int_{l_1^t} \left(\f{\sqrt{1+\kappa^2}\e}{2\kappa} |\pa_y u|^2+\f{\kappa}{\sqrt{1+\kappa^2}\e} W(u)  \right)\,dydx\\
        \geq & \f{1}{\sqrt{1+\kappa^2}} \sigma (1-c_0\e-y^*)+\f{\kappa}{\sqrt{1+\kappa^2}}\mathcal{H}^1(A)C_1\\
        =& \sqrt{\sigma^2(1-c_0\e-y^*)^2+(C_1\mathcal{H}^1(A))^2}.
    \end{split}
\end{equation}

From \eqref{est on Om1: new} we can update \eqref{est of y*: lower bound} to be 
\beqo
\begin{split}
    & \int_{B_1} \left( \f{\e}{2}|\na u|^2+\f{1}{\e} W(u)  \right)\\
    \geq & \sqrt{\sigma^2(1-c_0\e-y^*)^2+(C_1\mathcal{H}^1(A))^2}+\sqrt{3+4(y^*+\f12)^2}\sigma-C\e^{\f12} \\
    \geq& (1-y^*+\sqrt{3+4(y^*+\f12)^2})\sigma-C\e^{\f12}+\f{(C_1\mathcal{H}^1(A))^2}{2\sigma(1-c_0\e-y^*)},
\end{split}
\eeqo
which together with Lemma \ref{estimate of y*} and the upper bound \eqref{upper bound} implies 
\beq\label{est of H^1(A)}
\f{(C_1\mathcal{H}^1(A))^2}{2\sigma (1-c_0\e-y^*)}\leq C\e^{\f12}\; \Rightarrow \; \mathcal{H}^1(A)\leq C_2(\g,W) \e^{\f14}. 
\eeq

We also set
\beq\label{def of B}
B:=\{t: t\in [\sqrt3C_0\e^{\f14},\f12], \; \max\limits_{z\in l_2^t} |u(z)-a_1|>\g\}.
\eeq
An analogous computation implies 
\beq\label{est of H^1(B)}
\mathcal{H}^1(B)\leq C_2(\g,W)\e^{\f14}. 
\eeq

Set 
\beqo
C_3(\g,W):=\sqrt3 C_0+3C_2,\qquad \e(\g,W)=\f{1}{16C_3^4}.
\eeqo
We fix a small $\e< \e(\g,W)$.  Then from \eqref{est of H^1(A)} and \eqref{est of H^1(B)} it follows that  there exists $t_0\in [\sqrt3C_0\e^{\f14}, C_3\e^{\f14}]$ (note that $(\sqrt{3}C_0+3C_2)\e^{\f14}<\f12$ by the choice of $\e$) such that 
\beqo
 |u(z)-a_1|\leq \g,\quad \forall z\in l_1^{t_0}\cup l_2^{t_0}.
\eeqo

Let $D^1_\g$ denote the region enclosed by $l_1^{t_0}$, $l_2^{t_0}$ and $\pa B_1$. It follows that 
\beqo
|u(z)-a_1|\leq \g,\quad \forall z\in \pa D^1_\g.
\eeqo

By the variational maximum principle \cite{AF2} (cf. Theorem 4.1 in \cite{afs-book}), there exists $\g_0$, which only depends on $W$, such that if $\g<\g_0$ and $u$ satisfies
\begin{equation*}
\begin{split}
   &u \text{ minimizes }\int_{D^1_\g} \left( \f{\e}{2}|\na u|^2+\f{1}{\e}W(u) \right)\,dz,\\
   &\quad |u(z)-a_1|\leq \g \ \text{ on }\pa D^1_\g,
\end{split}
\end{equation*}
then 
\begin{equation}\label{maximum principle}
    |u(z)-a_1|\leq \g\ \text{ on }D^1_\g.
\end{equation}
This further implies that the diffused interface $I_{\e,\g}\cap D_1$ is contained in a $C_3\e^{\f14}$ neighborhood of $T$. The same argument works for $I_{\e,\g}\cap D_j$ for $j=2,3$ and we conclude the proof of the first part \eqref{diffused interface close to Gamma}.

Finally from \eqref{maximum principle} and linear elliptic theory the exponential decay estimate \eqref{distance to a_i} follows, which completes the proof. 

\end{proof}

\section{Width of the transition layer}\label{sec:width}

Set 
\begin{equation*}
    \G_{\e,\g}^i:= \{z\in\bar{B}_1:\ |u_\e(z)-a_i|=\g\}.
\end{equation*}
Then by definition of the diffused interface $I_{\e,\g}$  we have
\beqo
\pa I_{\e,\g}\subset \bigcup\limits_{i=1}^3 \G_{\e,\g}^i.
\eeqo

The following result shows the ``width" of $I_{\e,\g}$ is controlled by $O(\e)$. 
\begin{proposition}\label{width of diffused interface}
Fix $\g< \min\{\g_0,\min\limits_{i\neq j\in\{1,2,3\}} \f12 |a_i-a_j|\}$. There exists a constant $C=C(\g,W)$ such that for any $i\in \{1,2,3\}$ and sufficiently small $\e$, 
\begin{equation}\label{distance of G^i to G^j}
    \G^i_{\e,\g}\subset N_{C\e}(\bigcup\limits_{j\neq i}\G^i_{\e,\g})=\{z: \dist(z,\bigcup\limits_{j\neq i}\G^i_{\e,\g})\leq C\e\}.
\end{equation}
\end{proposition}
\begin{proof}
For simplicity we write $u_\e=u$. Without loss of generality, we prove \eqref{distance of G^i to G^j} for $i=1$. We pick an arbitrary point $z_1\in \G^1_{\e,\g}$ and set
\beqo
r_1:=\max \left\{r: B(z_1,r)\cap \left(  \G^2_{\e,\g}\bigcup \G^3_{\e,\g} \right)=\emptyset\right\}.
\eeqo
Since $\g<\min\limits_{i\neq j\in\{1,2,3\}} \f12 \vert a_i-a_j\vert $, we have $\vert u(z_1)-a_j\vert>\g$ for $j=2,3$. Then from the definition of $r_1$ it is not hard to show that
\beqo
\vert u(z)-a_j\vert\geq \g, \quad j=2,3,\ z\in B(z_1,r_1).
\eeqo

We define the blow up map $v_\e(\zeta):=u_\e(z_1+\e \zeta)$. Then $v_\e$ satisfies
\begin{gather*}
        v_\e  \text{ minimizes } \int_{\Om/\e} \left(\f12\vert\na u\vert^2+W(u)\right)\,d \zeta,\\
         \vert v_\e(0)-a_1\vert =\g,\\
         \vert v_\e(\z)-a_j\vert \geq \g,\quad j=2,3,\ \z\in B(0,\f{r_1}{\e}),\\
         \vert \na v_\e\vert \leq M,\quad M \text{ is a constant depending on } W.
\end{gather*}
Here the gradient bound follows from the smoothness assumption of $W$ and standard elliptic regularity theory. This implies 
\beqo
\vert v_\e(\z)-a_1\vert \geq \f{\g}{2},\quad \z\in B(0, \f{\g}{2M}).
\eeqo
Then by the vector version of the Caffarelli-C\'{o}rdoba density estimate \cite{AF3} (cf. Theorem 5.2 in \cite{afs-book}), there exists a positive constant $C=C(\g,W)$ such that
\beqo
\mathcal{L}^2\left(\{\vert v_\e(\z)-a_1\vert \geq \f{\g}{2}\}\cap B(0,\f{r_1}{\e})\right)\geq C\left(\f{r_1}{\e}\right)^2.
\eeqo
It follows that 
\beqo
\mathcal{L}^2\left(\{\vert v_\e(\z)-a_i\vert \geq \f{\g}{2},\ \forall i=1,2,3\}\cap B(0,\f{r_1}{\e})\right)\geq C\left(\f{r_1}{\e}\right)^2.
\eeqo

Then we compute
\beq\label{energy in Br1z1}
\begin{split}
    \int_{B(z_1,r_1)}\f{1}{\e} W(u)\,dz&=\e\int_{B(0,\f{r_1}{\e})}W(v_\e)\,d\z\\
    &\geq \e \f{c_W}{2} (\f{\g}{2})^2\cdot C (\f{r_1}{\e})^2\\
    &=C(\g,W)\f{r_1^2}{\e},
\end{split}
\eeq

On the other hand, we can derive an upper bound for the energy of $u$ inside $B(z_1,r_1)$ by constructing an energy competitor $v$ such that $v=u$ in $\Om\setminus B(z_1,r_1)$. We can assume $r_1\geq 2\e$, otherwise there is nothing to prove. Let
\begin{equation}\label{def of energy competitor v}
    v(z)=\begin{cases} u(z), & \vert z-z_1\vert \geq r_1,\\
    \f{r_1-\vert z-z_1\vert }{\e} a_1+\f{\vert z-z_1\vert -r_1+\e}{\e} u(z), &r_1-\e\leq \vert z-z_1\vert <r_1,\\
    a_1,& \vert z-z_1\vert <r_1-\e
    \end{cases}
\end{equation}

We compute the energy of $v$ inside $B(z_1,r_1)$.
\begin{equation}\label{energy in Br1z1 v}
    \begin{split}
       & \int_{B(z_1,r_1)}\f{\e}{2}\vert \na v\vert ^2+\f{1}{\e}W(v)\,dz\\
        =& \int_{r_1-\e\leq \vert z-z_1\vert <r_1} \f{\e}{2}\left\vert \na u(z)+ \na \left(\f{\vert z-z_1\vert -r_1}{\e}(u(z)-a_1)\right)\right\vert ^2\,dz\\
        &\qquad + \int_{r_1-\e\leq \vert z-z_1\vert <r_1} \f{1}{\e} W(v)\,dz\\
        \leq & \int_{r_1-\e\leq \vert z-z_1\vert <r_1} \f{\e}{2}\left\{ 4\vert \na u(z)\vert ^2+2\vert u(z)-a_1\vert ^2 \f{\vert \na \vert z-z_1\vert \vert ^2}{\e^2}  \right\}\,dz\\
        &\qquad +\int_{r_1-\e\leq \vert z-z_1\vert <r_1} \f{1}{\e} W(v)\,dz\\
        \leq &  (2\pi r_1\cdot\e)\cdot\f{\e}{2}\cdot (4\left\vert \frac{M}{\e}\right\vert ^2+2\left\vert \f{M}{\e}\right\vert ^2)+   (2\pi r_1\cdot\e)\cdot\f{c}{\e}\\
        \leq &C(M) r_1.
    \end{split}
\end{equation}
Here in the calculation we have used that $\vert u(z)\vert \leq M,\ \vert \na u(z)\vert \leq \f{M}{\e}$ and that $W(v)$ is also uniformly bounded by some constant $c$. From \eqref{energy in Br1z1}, \eqref{energy in Br1z1 v} and the minimality of $u$ we obtain
\begin{align*}
    C(\g,W)\f{r_1^2}{\e}\leq C(M)r_1\ \Rightarrow\ r_1\leq C\e, 
\end{align*}
where the last constant $C$ only depends on $\g$ and $W$. The proof of Proposition \ref{width of diffused interface} is complete. 
\end{proof}

\section{Proof of Theorem \ref{main theorem}}\label{sec:proof of main theorem}

Throughout the section, we denote by $C_0=C(\g,W)$ the constant defined in Proposition \ref{width of diffused interface}. Fix a small $\g$ and a small constant $\e<\e(\g,W)$. We always assume $\g,\e$ are suitably small constants, and the requirement on the smallness of these values will be explained along the proof.  According to Sard's Theorem and the Implicit Function Theorem, we can without loss of generality assume that each connected component of $\G^i_{\e,\g}\cap B_1$ is a $C^1$ curve\footnote{For fixed $\e$, from Sard's Theorem and the Implicit Function Theorem we have for a.a.$\g$, each connected component of $\G^i_{\e,\g}\cap B_1$ is a $C^1$ curve. Therefore even if the chosen $\{\e,\g\}$ doesn't satisfy this property, we can always take a slightly smaller $\g'$ so that this property holds for $\{\e,\g'\}$. Then one can proceed with the rest of the proof using $\{\e,\g'\}$ and all the conclusions will not be affected.}. 

For fixed small $\e,\g$, let $\hat{D}_{\e,\g}^1$ denote the connected component of the set $\{z:\vert u_\e(z)-a_1\vert <\g\}$ that connects to the boundary $\pa B_1$. $\hat{D}^1_{\e,\g}$ satisfies the following properties
\begin{enumerate}
    \item  $\{z\in \pa B_1: u_\e(z)=a_1\}\subset \pa \hat{D}_{\e,\g}^1$,
    \item $\hat{D}^1_{\e,\g}$ is simply connected by the variational maximum principle argument (cf. \cite[Theorem 4.1]{afs-book}).
    \item By Proposition \ref{location of diffused interface}, $(\pa \hat{D}^1_{\e,\g}\setminus \pa B_1) \subset N_{C\e^{\f14}}(T)$, which is an $O(\e^{\f14})$ neighborhood of the triod. This property holds only if $\e<\e(\g,W)$ where $\e(\g,W)$ is the constant in \eqref{diffused interface close to Gamma}.
\end{enumerate}

We set 
\beq\label{def: hg}
\hat{\G}^1_{\e,\g}:=\text{ closure of }(\pa \hat{D}_{\e,\g}^1\setminus \pa B_1).
\eeq
Then by the definition of $g_\e$, $\hg$ is a continuous curve that intersects with $\pa B_1$ at only two points. We set $A$ as the intersection of $\pa B_1$ and $\hg$ that is $O(\e)$ close to the point $(0,1)$ and $B$ as the other intersection.

We first prove the following lemma that is a direct consequence of the upper bound Lemma \ref{lemma:upper bound} and Proposition \ref{width of diffused interface}.

\begin{lemma}\label{claim}
Fix $\g<\min\{\g_0,\min\limits_{i,j}\vert a_i-a_j \vert, \sqrt{\f{\s}{20C_W}}\}$. For any $n>0$, there exists $\e_0:=\e_0(n,\g,W)$ such that for any $\e<\e_0$, there exists $y_0\in [\f14,\f38]$ that satisfies
\beqo
\forall\ z=(x,y)\in \hg\cap\{y\in [y_0-nC_0\e,y_0+nC_0\e]\}, \text{ it holds } \dist(z,\G^2_{\e,\g})\leq \dist(z,\G^3_{\e,\g}).
\eeqo
\end{lemma}

\begin{rmk}
It is not hard to see from the proof below that the constants $\f14,\,\f38$ in the above lemma can be replaced by any $0<r_1<r_2<1$. Here we take these values for simplicity of presentation.
\end{rmk}

\begin{proof}

First assume $\e$ is sufficiently small and fix $n$. We argue by contradiction. If for any $h\in[\f14,\f38]$, there exists $z(h)=(x,y)$ satisfying $\vert y-h\vert\leq nC_0\e$ and $\dist(z(h),\G^2_{\e,\g})> \dist(z(h),\G^3_{\e,\g})$, then by Proposition 5.1
\beqo
\dist(z(h),\G^3_{\e,\g})\leq C_0\e,
\eeqo
which implies
\beq\label{def:xi(h)}
\exists\  \xi(h)\in \G^3_{\e,\g},\quad \dist(\xi(h),z(h))\leq C_0\e. 
\eeq
Since $\vert \na u_\e\vert \leq \f{M}{\e}$, for any $z$ such that $\dist(z,\xi(h)) \leq \f{\g\e}{M}$ it holds that 
\beqo
\vert u_\e(z)-a_3\vert \leq 2\g.
\eeqo

Set $N=[\f{1}{32(n+1)C_0\e}]$ as the greatest integer less than or equal to $\f{1}{32(n+1)C_0\e}$. Here we require $\e\leq \f{1}{128(n+1)C_0}$ such that $N\geq 4$ and $N$ satisfies 
\beq\label{estimate of N}
\f{1}{64(n+1)C_0\e}< N \leq \f{1}{32(n+1)C_0\e}.
\eeq
Then we split the interval $[\f14,\f38]$ into $N$ small intervals, which are denoted by $I_1$, $I_2$,...,$I_N$:
\beqo
\begin{split}
   &I_j=[\f14+4(j-1)(n+1)C_0\e,\ \f14+4j(n+1)C_0\e),\quad 1\leq j\leq N-1,\\
   &I_N:=[\f14+4(N-1)(n+1)C_0\e,\f38].
\end{split}
\eeqo

From the definition we know that $\cl(I_j)=4(n+1)C_0\e$ for $j=1,...,N-1$ and $4(n+1)C_0\e\leq \cl(I_N)<8(n+1)C_0\e$. We define
\beqo
h_j:=\f14+(4j-2)(n+1)C_0\e,\quad j=1,...,N.
\eeqo
Then $h_j$ is in $I_j$ for $j=1,...N$ and $\vert h_N-\f38\vert \geq 2(n+1)C_0\e$. 

For any $j=1,...,N$, by the assumption above there are points $z(h_j)=(z_1,z_2)$, $\xi(h_j)=(\xi_1,\xi_2)$ such that 
\beqo
\vert z_2-h_j\vert \leq nC_0\e,\ \ \vert \xi_2-z_2\vert \leq C_0\e,\ \ z(h_j)\in\hg,\ \ \xi(h_j)\in\G^3_{\e,\g}.
\eeqo
Hence for any $h\in (\xi_2-\f{\g\e}{M},\xi_2+\f{\g\e}{M})$,
\beqo
\exists \hat{z}(h)=(\hat{x}(h),h),\ s.t.\ \ \vert u_\e(\hat{z}(h))-a_3\vert \leq 2\g.
\eeqo

We define the set of points with such property
\beqo
K:=\{h\in[\f14,\f38]: \exists \hat{z}(h)=(\hat{x}(h),h),\ s.t.\ \vert u_\e(\hat{z}(h))-a_3\vert \leq 2\g\}.
\eeqo
The deduction above implies that for any $j=1,...,N$, 
\beq\label{measure of K}
\cl(K\cap I_j)\geq \min \{\f{2\g\e}{M}, 2(n+1)C_0\e\}= C_1\e, 
\eeq
where $C_1$ is a constant depending on $n,\,\g,\,W$. 

For $h\in K$, $u_\e$ equals to $a_1$ at $(-\sqrt{1-h^2},h)$ and $a_2$ at $(\sqrt{1-h^2},h)$ respectively, while it is close to $a_3$ at some point $\hat{z}(h)$ in the middle. We compute (as before we write $u_\e=u$ for simplicity)
\begin{equation*}
    \begin{split}
        &\int_{\{y=h\}} \left(\f{\e}{2}\vert \pa_x u\vert ^2+\f{1}{\e}W(u)\right)\,dx\\
        \geq & \int_{-\sqrt{1-h^2}}^{\hat{x}(h)}\left(\f{\e}{2}\vert \pa_x u\vert ^2+\f{1}{\e}W(u)\right)\,dx+ \int^{\sqrt{1-h^2}}_{\hat{x}(h)}\left(\f{\e}{2}\vert \pa_x u\vert ^2+\f{1}{\e}W(u)\right)\,dx\\
        \geq & 2(\s-4C_W\g^2)>\f32\s,
    \end{split}
\end{equation*}
where in the last inequality we further require that $\g<\sqrt{\f{\s}{20C_W}}$. 

Note that when $\e$ is suitably small, $y_*\leq \f14$. We compute the energy in $\Om_1$,
\begin{equation}\label{proof of claim, energy in Om_1}
    \begin{split}
        &\int_{\Om_1\cap\{y\in[\f14,\f38]\}} \left(\f{\e}{2}\vert \pa_x u\vert ^2+\f{1}{\e}W(u)\right)\,dxdy\\
        \geq &\sum_j\int_{y\in I_j} \int_{-\sqrt{1-y^2}}^{\sqrt{1-y^2}} \left(\f{\e}{2}\vert \pa_x u\vert ^2+\f{1}{\e}W(u)\right)\,dxdy\\
        \geq& \sum_j \bigg{(} \int_{y\in I_j\cap K} \int_{-\sqrt{1-y^2}}^{\sqrt{1-y^2}} \left(\f{\e}{2}\vert \pa_x u\vert ^2+\f{1}{\e}W(u)\right)\,dxdy \\
        &\qquad \quad +\int_{y\in I_j\setminus K} \int_{-\sqrt{1-y^2}}^{\sqrt{1-y^2}} \left(\f{\e}{2}\vert \pa_x u\vert ^2+\f{1}{\e}W(u)\right)\,dxdy \bigg{)}\\
        \geq &\sum_j\left( \cl(I_j\cap K)\cdot \f32\s+\cl(I_j\setminus K)\s \right)\\
        = & \f{\s}{8}+\f{\s}{2}\left(\sum_j \cl(I_j\cap K) \right)\\
        \geq &  \f{\s}{8}+\f{\s}{2}(C_1\e) \f{1}{64(n+1)C_0\e}\\
        \geq & \f{\s}{8}+C_2\s.
    \end{split}
\end{equation}
Here $C_2=C_2(\g,W)$ is a positive constant which only depends on $n,\,\g,\,W$. Adding the energy in $\Om_1$ and $\Om_2$, we have
\begin{equation*}
    \begin{split}
        &\int_{B_1} \left(  \f{\e}{2}\vert \na u\vert ^2+\f{1}{\e} W(u) \right)\ dxdy\\
        =&\int_{\Om_1} \left(  \f{\e}{2}\vert \na u\vert ^2+\f{1}{\e} W(u) \right)\ dxdy+\int_{\Om_2} \left(  \f{\e}{2}\vert \na u\vert ^2+\f{1}{\e} W(u) \right)\ dxdy \\
        \geq & (3+C_2(n,\g,W))\sigma -C\e^{\f12}. 
    \end{split}
\end{equation*}
Here the constant $C$ in the last inequality is the same one as in Proposition 3.2.  Taking $\e$ such that $C\e^{\f12}\leq \f{C_2}{2}\s$ yields a contradiction with the upper bound, which proves the lemma.

\end{proof}

\begin{lemma}\label{lemma: three point}
Fix $\g< \min\{\g_0,\min\limits_{i,j\in\{1,2,3\}} \f12 \vert a_i-a_j\vert, \sqrt{\f{\sigma}{20C_W}}\}$. There exists a small constant $\e(\g,W)$ such that for any $\e<\e(\g,W)$, there are three points $P_\e,Q_\e,R_\e\in \overline{B}_{\f12}$ that satisfy
\begin{align*}
&\vert u_\e(P_\e)-a_1\vert = \g,\ \vert u_\e(Q_\e)-a_2\vert =\g,\  \vert u_\e(R_\e)-a_3\vert = \g,\\
& \max\{\dist(P_\e,Q_\e), \dist(P_\e,R_\e)\}\leq C_0\e.
\end{align*}
\end{lemma}

\begin{proof}
\begin{figure}[ht]
\centering
\begin{tikzpicture}[thick]
\draw [black!20!green,domain=0:85] plot ({3.5*cos(\x)}, {3.5*sin(\x)});
\draw [black!20!green,domain=335:360] plot ({3.5*cos(\x)}, {3.5*sin(\x)});
\draw [dotted,domain=85:95] plot ({3.5*cos(\x)}, {3.5*sin(\x)});
\draw [red,domain=95:205] plot ({3.5*cos(\x)}, {3.5*sin(\x)});
\draw [dotted,domain=205:215] plot ({3.5*cos(\x)}, {3.5*sin(\x)});
\draw [blue,domain=215:325] plot ({3.5*cos(\x)}, {3.5*sin(\x)});
\draw [dotted,domain=325:335] plot ({3.5*cos(\x)}, {3.5*sin(\x)});
\draw[red] (-0.6, {sqrt(3.5*3.5-0.6*0.6)})--(-0.6,{0.6/sqrt(3)})--({-0.3-sqrt(3)*sqrt(11.89)/2},{-sqrt(11.89)/2+0.3*sqrt(3)});
\draw[rotate around={120:(0,0)},color=blue] (-0.6, {sqrt(3.5*3.5-0.6*0.6)})--(-0.6,{0.6/sqrt(3)})--({-0.3-sqrt(3)*sqrt(11.89)/2},{-sqrt(11.89)/2+0.3*sqrt(3)});
\draw[rotate around={240:(0,0)},color=green] (-0.6, {sqrt(3.5*3.5-0.6*0.6)})--(-0.6,{0.6/sqrt(3)})--({-0.3-sqrt(3)*sqrt(11.89)/2},{-sqrt(11.89)/2+0.3*sqrt(3)});
\fill[color=black!10] ({-0.3-sqrt(3)*sqrt(11.89)/2},{-sqrt(11.89)/2+0.3*sqrt(3)})--(-0.6,{0.6/sqrt(3)}) --(-0.6, {sqrt(3.5*3.5-0.6*0.6)}) arc (100:90:3.5) --(0,3.5)--(0,0)--({-1.75*sqrt(3)},-1.75);
\fill[rotate around={120:(0,0)}, color=black!10] ({-0.3-sqrt(3)*sqrt(11.89)/2},{-sqrt(11.89)/2+0.3*sqrt(3)})--(-0.6,{0.6/sqrt(3)}) --(-0.6, {sqrt(3.5*3.5-0.6*0.6)}) arc (100:90:3.5) --(0,3.5)--(0,0)--({-1.75*sqrt(3)},-1.75);
\fill[rotate around={240:(0,0)}, color=black!10] ({-0.3-sqrt(3)*sqrt(11.89)/2},{-sqrt(11.89)/2+0.3*sqrt(3)})--(-0.6,{0.6/sqrt(3)}) --(-0.6, {sqrt(3.5*3.5-0.6*0.6)}) arc (100:90:3.5) --(0,3.5)--(0,0)--({-1.75*sqrt(3)},-1.75);

\draw[color=red, rounded corners=6pt] (-0.3,{sqrt(12.16)}) node[above, color=black]{A}--(-0.5,2.5)--(-0.1,1.2)--(-0.3,0)--(-1.45,-0.4)--(-2.35, -1.2)--(-3.18, -1.47) node[below, color=black]{B};
\draw[->](-0.25,0.6)--(-0.8,0.6) node[left]{$\hg$};
\filldraw(-0.15,1.3) circle (1.5pt) node [above right]{$\hat{z}_1$ is closer to $\G^2_{\e,\g}$};
\filldraw(-1.41,-0.4) circle (1.5pt) node [ right]{$\hat{z}_2$ is closer to $\G^3_{\e,\g}$};
\draw[thin,->,color=black!70] (0.5,1)--(3.9,1);
\draw[red](4.2,2.2)--(4.2,-0.2);
\draw[dashed] (4.2,-0.2)--(8,-0.2);
\draw[green] (8,-0.2)--(8,2.2) node [right,color=black] at (8,1) {$y\in I_j$};
\draw[dashed] (4.2,2.2)--(8,2.2);
\fill[color=black!10] (4.2,2.2)--(4.2,-0.2)--(8,-0.2)--(8,2.2)--(4.2,2.2);
\draw[color=red] (4.8,-0.2)--(5.2,2.2);
\filldraw(5,1) circle (1.5pt) node [above left] {$z(h_j)$};
\fill[color=blue!50] (5.8,0.5) circle (0.5);
\filldraw(5.8,0.5) circle (1.5pt) node [below right] {$\xi(h_j)$};

\end{tikzpicture}
\caption{$\hat{z}_1,\hat{z}_2\in\hg$ and satisfy $\dist(\hat{z}_1,\G^3_{\e,\g})\geq \dist(\hat{z}_1,\G^2_{\e,\g})$, $\dist(\hat{z}_2,\G^2_{\e,\g})\geq \dist(\hat{z}_2,\G^3_{\e,\g})$ respectively. The figure on the right shows the enlargement of one portion of the transition layer between $a_1,a_2$. In an $O(\e)$ neighborhood of $\xi(h_j)$ (the blue region), $\vert u(z)-a_3\vert \leq 2\g$.}
\label{pic3}
\end{figure}

\vspace{5pt}
From Lemma \ref{claim} (taking $n=1$) and Proposition \ref{width of diffused interface} there is at least one point $\hat{z}_1=(x_1,y_1)\in \hg$ which satisfies the following
\begin{enumerate}
    \item $y_1\in [\f14,\f38]$,
    \item $\dist(\hat{z}_1,\G^2_{\e,\g})\leq C_0\e$, $\ \dist(\hat{z}_1,\G^3_{\e,\g})\geq  \dist(\hat{z}_1,\G^2_{\e,\g})$,
    \item $\hg(\hat{z}_1,B)\cap \{(x,y)\in B_1:y>y_1\}=\emptyset$, here $\hg(\hat{z}_1,B)$ denotes the part of $\hg$ that connects $\hat{z}_1$ and $B$. 
\end{enumerate}
Here we have the third property because $\hat{z}_1$ is chosen from the intersections of $\hg$ and the horizontal line $\{y=y_1\}$ where $y_1$ satisfies Lemma \ref{claim}. Therefore we can choose the ``last" intersection of $\hg$ and $\{y=y_1\}$ such that the third property holds. 

Similarly we can apply the same argument to the transition layer between $a_1$ and $a_3$ and find one point $\hat{z}_2=(x_2,y_2)\in\hg$ satisfying 
\begin{enumerate}
    \item $\sqrt{3}x_2+y_2\in [-\f34,-\f12]$,
    \item $\dist(\hat{z}_2,\G^3_{\e,\g})\leq C_0\e$, $\ \dist(\hat{z}_2,\G^2_{\e,\g})\geq  \dist(\hat{z}_2,\G^3_{\e,\g})$,
    \item $\hg(A,\hat{z}_2)\cap \{(x,y)\in B_1: \sqrt{3}x+y<\sqrt{3}x_2+y_2\}=\emptyset$. 
\end{enumerate}

From the path-connectedness of $\hg$ it follows that there is a continuous map $\hat{f}:[0,1]\ri\hg$ that satisfies $\hat{f}(0)=\hat{z}_1$, $\hat{f}(1)=\hat{z}_2$. By the property (3) of $\hat{z}_1$ and $\hat{z}_2$ and simple geometry arguments, we know that 
\beq\label{the chosen part is away from boundary}
\hg(\hat{z}_1,\hat{z}_2) \subset \overline{B}_{\f12}\cap N_{C\e^{\f14}}(T), \ \ \text{when }C\e^{\f14}<\f18. 
\eeq
Define the function 
\beqo
\al:[0,1]\ri\BR,\quad \al(t)=\dist(\hat{f}(t),\G^2_{\e,\g})-\dist(\hat{f}(t),\G^3_{\e,\g}).
\eeqo
Due to the continuity of $\hat{f}$ and the distance function to a closed set, $\al(t)$ is also continuous. The properties of $\hat{z_1}$ and $\hat{z}_2$ mean that $\al(0)\leq 0$ and $\al(1)\geq 0$. Then the intermediate value theorem immediately implies there exists $t_0\in [0,1]$ such that $\al(t_0)=0$. Set $P_\e=\hat{f}(t_0)$. From Proposition \ref{width of diffused interface} it follows that  
\beqo
\dist(P_\e,\G^2_{\e,\g})=\dist(P_\e,\G^3_{\e,\g})\leq C_0\e.
\eeqo
Take $Q_\e\in \G_{\e,\g}^2$ and $R_\e\in\G^3_{\e,\g}$ such that 
\begin{equation*}
\dist(Q_\e,P_\e)=\dist(R_\e,P_\e)=\dist(P_\e,\G^2_{\e,\g})=\dist(P_\e,\G^3_{\e,\g})\leq C_0\e.
\end{equation*}
Then it is straightforward to verify that $P_\e$, $Q_\e$ and $R_\e$ satisfy all the properties in the statement of Lemma \ref{lemma: three point}. The proof is complete. 
\end{proof}

Next we present the key lemma for the main Theorem \ref{main theorem}, which has a more discrete nature.

\begin{lemma}\label{New ver of lemma 6.2}
Fix $\g<\min\{\g_0,\min\limits_{i,j}\vert a_i-a_j \vert, \sqrt{\f{\s}{20C_W}}\}$. For any $k\in \mathbb{N}^+$, there exists $\e_0(k,\g,W)$ such that for any $\e<\e_0$, there is a point $P_\e\in\hg\cap B_{\f12}$ that satisfies the following,
\begin{align*}
    &\exists\ \{Q^j_k\}_{j=1}^{2^k}\subset \G^2_{\e,\g},\ \{R^j_k\}_{j=1}^{2^k}\subset \G^3_{\e,\g} \text{ such that }\\
    & \dist(Q^j_k,P_\e)\leq (32j+1)C_0\e,\quad \dist(R^j_k,P_\e)\leq (32j+1)C_0\e,\\
    &\text{For any two points }A_1,\,A_2 \in \{Q^j_k\}_{j=1}^{2^k}\cup \{R^j_k\}_{j=1}^{2^k}, \ \ \dist(A_1,A_2)\geq 6C_0\e,\\
    &\text{For any point }A\in \{Q^j_k\}_{j=1}^{2^k}\cup \{R^j_k\}_{j=1}^{2^k},\ \ \dist(A,\hg)\leq C_0\e.
\end{align*}
\end{lemma}

\begin{proof}
Assume $\g$ and $\e$ are suitably small. The specific conditions on the smallness of $\g,\,\e$ will be provided along the proof. We focus on the curve $\hg\subset \G^1_{\e,\g}$, which is a $C^1$ curve that intersects with the boundary $\pa B_1$ at two points $A,\, B$.

First of all we parameterize $\hg$ by arclength: 
\beqo
\hg=\{\eta(t):t\in[0,L],\, \eta(0)=A,\, \eta(L)=B,\, \vert\eta'(t)\vert=1 \text{ a.e. }\},
\eeqo
where $L=\mathcal{H}^1(\hg)$ is the length of the curve $\hg$. 

\begin{step}
``Discretize the curve $\hg$". We can find a sequence of points $z_1,...,z_N$ by the following rule:
\begin{enumerate}
    \item $z_1=A=\eta(0)$, $t_1=0$;
    \item If we already define $z_i=\eta(t_i)$, then take 
    \beq\label{find t_{i+1}}
    t_{i+1}=\sup\limits_{t\geq t_i} \{t: \dist(\eta(t),z_i)\leq 8C_0\e\},\quad z_{i+1}=\eta(t_{i+1});
    \eeq
    \item When $\dist(B,z_N)\leq 8C_0\e$ for some $N$, the process stops.
\end{enumerate}
By the continuity of $\eta(t)$, we have immediately 
\beq\label{distance of zi zi+1}
\dist(z_{i+1},z_i)=8C_0\e.
\eeq
Since $\vert \eta'(t) \vert=1$ a.e., it holds 
\beqo
t_{i+1}-t_i\geq 8C_0\e.
\eeqo
Hence the process will stop in finite steps and 
\beqo
N\leq \f{L}{8C_0\e}+1.
\eeqo
Moreover for any $1\leq i\neq j\leq N$,
\beq\label{distance of zi zj}
\dist(z_i,z_j)\geq 8C_0\e
\eeq
Otherwise, if there exists $i>j$ such that $\dist(z_i,z_j)<8C_0\e$, by \eqref{distance of zi zi+1} we know $i\geq j+2$. Then  $\dist(\eta(t_i),z_j)<8C_0\e$ and $t_i>t_{j+1}$ together yield a contradiction with the choice of $t_{j+1}$.

Next we classify $\{z_i\}_{i=1}^N$ into two subsets according to their relative distances to $\G^2_{\e,\g}$ and $\G^3_{\e,\g}$.
\begin{align*}
    &\mathbb{Z}_2:=\{z_i:\dist(z_i,\G^2_{\e,\g})\leq \dist(z_i,\G^3_{\e,\g})\},\\
    &\mathbb{Z}_3:=\{z_i\}_{i=1}^N\setminus \mathbb{Z}_2.
\end{align*}
\end{step}

\begin{step}
Fix $k\in \mathbb{N}^+$. Taking $n=2^{k+5}$ in  Lemma \ref{claim}, we have that when $\e<\e(n,\g,W)$, there exists a $\tilde{y}\in[\f14,\f38]$ such that for any $z_i=(x_i, y_i)$ satisfying $\vert y_i-\tilde{y}\vert\leq  2^{k+5}C_0\e$, $z_i\in \mathbb{Z}_2$. Define
\beqo
t_{i_0}:=\sup\{t_i:z_i=\eta(t_i)=(x_i,y_i),\ y_i\geq \tilde{y}\}.
\eeqo
This definition of $t_{i_0}$ and \eqref{distance of zi zi+1} implies that 
\beqo
\begin{split}
&\vert y_{i_0}-\tilde{y}\vert\leq 8C_0\e,\\
& y_i<\tilde{y},\ \forall i>i_0.
\end{split}
\eeqo
In particular we have for all $i_0\leq i\leq i_0+2^{k+1}-1$ it holds that 
\beqo
\vert y_i-\tilde{y}\vert\leq 2^{k+4}C_0\e, \quad y_i\in[\f18,\f12],\quad z_i\in \mathbb{Z}_2.
\eeqo
Here we only require $2^{k+4}C_0\e\leq \f18$.

Following the same argument, we can also get a $j_0$ that satisfies the following,
\begin{align}
    \nonumber &j_0>i_0;\\
    \label{location of zj between i0 and j0}&\forall\  i_0< j\leq j_0, \quad \sqrt{3}x_j+y_j> -\f34,\ y_j<\f38;\\
    \nonumber &\forall\ j_0-2^{k+1}+1\leq j\leq j_0,\quad z_j\in \mathbb{Z}_3,\ \sqrt{3}x_j+y_j\in(-1,-\f14).
\end{align}

The locations of $\{z_i\}_{i=i_0}^{i_0+2^{k+1}-1}$ and $\{z_j\}_{j=j_0-2^{k+1}+1}^{j_0}$ imply that they are disjoint subsets. Since $j_0>i_0$, we further infer that 
\beqo
j_0-2^{k+1}+1>i_0+2^{k+1}-1.
\eeqo
Furthermore, \eqref{location of zj between i0 and j0} implies 
\beq\label{location of zj between i0 and j0 2}
z_i\in B_{\f12},\quad \forall\  i_0\leq j\leq j_0.
\eeq
\end{step}

\begin{step}
For $i_0\leq l\leq j_0-2^{k+1}+1$, define
\beqo
N(l):=\vert \{ z_l,z_{l+1},...,z_{l+2^{k+1}-1} \}\cap\mathbb{Z}_2 \vert
\eeqo
as the number of points in $\{z_i\}_{i=l}^{l+2^{k+1}-1}\cap \mathbb{Z}_2$. Then 
\beqo
N(i_0)=2^{k+1},\quad N(j_0-2^{k+1}+1)=0.
\eeqo
Also each time $l$ is shifted by $1$, $N(l)$ is changed at most by $1$, i.e.
\beqo
\vert N(l+1)-N(l) \vert \leq 1.
\eeqo
Hence by continuity there exists a $l_0$ such that 
\beqo
N(l_0)=2^k, \quad i_0<l_0<j_0-2^{k+1}+1.
\eeqo

Next we can find $l_1\in \{l_0,...,l_0^{2^k}\}$ such that $\vert \{z_i\}_{i=l_1}^{l_1+2^k-1}\cap\mathbb{Z}_2   \vert=2^{k-1} $. Indeed, if $\vert \{z_i\}_{i=l_0}^{l_0+2^k-1}\cap\mathbb{Z}_2   \vert=2^{k-1}$, then we simply take $l_1=l_0$. Otherwise without loss of generality we assume $\vert \{z_i\}_{i=l_0}^{l_0+2^k-1}\cap\mathbb{Z}_2   \vert>2^{k-1}$, then $\vert \{z_i\}_{i=l_0+2^k}^{l_0+2^{k+1}-1}\cap\mathbb{Z}_2   \vert<2^{k-1}$. By continuity we can find $l_1\in \{l_0+1,...,l_0+2^k-1\}$ such that 
\beqo
\vert \{z_i\}_{i=l_1}^{l_1+2^k-1}\cap\mathbb{Z}_2   \vert=2^{k-1}.
\eeqo

Proceeding in the same way we finally get $l_0,l_1,...,l_k$ that satisfy
\begin{enumerate}
    \item $l_j\leq l_{j+1}\leq l_{j}+2^{k-j}-1$, for $j=0,...,k-1$.
    \item $\vert \{z_i\}_{i=l_j}^{l_j+2^{k+1-j}-1}\cap \mathbb{Z}_2 \vert=\vert \{z_i\}_{i=l_j}^{l_j+2^{k+1-j}-1}\cap \mathbb{Z}_3 \vert=2^{k-j}$ for $j=0,...,k$.
    \item $\{z_i\}_{i=l_j}^{l_j+2^{k+1-j}-1}\subset \overline{B}(z_{l_k}, 2^{k+1-j}\cdot 8C_0\e)$ for $j=0,...,k$.
\end{enumerate}

Take $P_\e=z_{l_k}$, then from the above construction we have that
\beqo
\begin{split}
 &\vert\overline{B}(P_\e,2^{j+3}C_0\e)\cap \mathbb{Z}_2\vert\geq 2^{j-1},\\
 &\vert\overline{B}(P_\e,2^{j+3}C_0\e)\cap \mathbb{Z}_3\vert\geq 2^{j-1}.  
\end{split}
\eeqo
Recall the definition of $\mathbb{Z}_2$, if $z_i\in \mathbb{Z}_2$, then there exists a $Q_i\in \G^2_{\e,\g}$ such that $\dist(z_i,Q_i)\leq C_0\e$. For $z_i\neq z_j\in \mathbb{Z}_2$, let $Q_i,Q_j$ be the corresponding points on $\G^2_{\e,\g}$. Then $Q_i,Q_j$ satisfy $\dist(Q_i,Q_j)\geq 6C_0\e$.

Therefore, we have obtained $\{Q_1,...,Q_{2^k}\}\subset \G^2_{\e,\g}$, $\{R_1,...,R_{2^k}\}\subset \G^3_{\e,\g}$
satisfying
\beqo
\begin{split}
&\{Q_1,...,Q_{2^j},R_1,...,R_{2^j}\}\subset \overline{B}(P_\e,(2^{j+4}+1)C_0\e),\quad \forall\  0\leq j\leq k,\\
&\forall\ A_1,\, A_2\subset \{Q_1,...,Q_{2^j},R_1,...,R_{2^j}\}, \ \dist(A_1,A_2)\geq 6C_0\e.
\end{split}
\eeqo

Let $Q^j_{k}=Q_j$, $R^i_k=R_j$. The proof of the lemma is complete.
\end{step}
\end{proof}

\begin{proof}[Conclusion of the proof of Theorem \ref{main theorem}]

By Lemma \ref{New ver of lemma 6.2}, we can find a sequence $\e_k\ri 0$ and a sequence of points $P_{\e_k}=(x_k,y_k)\in\hat{\G}^1_{\e_k,\g}\cap B_{\f12}$ such that there exist $\{Q^j_{\e_k}\}_{j=1}^{2^k}\cap \{R^j_{\e_k}\}_{j=1}^{2^k}$ satisfying that the pairwise distance is larger than $6C_0\e_k$ and 
\begin{equation*}
    \begin{split}
       &Q_{\e_k}^j\in \G^2_{\e_k,\g},\  R_{\e_k}^j\in \G^3_{\e_k,\g},\\
       &\dist(Q_{\e_k}^j,\G^1_{\e_k,\g})\leq C_0\e_k,\ \dist(R_{\e_k}^j,\G^1_{\e_k,\g})\leq C_0\e_k,\\
       & \dist(Q^j_{\e_k},P_{\e_k})\leq (32j+1)C_0\e_k,\quad \dist(R^j_{\e_k},P_{\e_k})\leq (32j+1)C_0\e_k.
    \end{split}
\end{equation*}
Moreover, since $\hat{\G}_{\e_k,\g}^1$ is connected, for any $r\in(0,\dist(P_{\e_k},\pa B_1) ]$ we have
\beqo
\hat{\G}_{\e_k,\g}^1\cap \pa B(P_{\e_k},r)\neq \emptyset.
\eeqo

We rescale $u_{\e_k}$ at the point $P_{\e_k}$:
\beq\label{rescale u_e}
    U_{k}(X,Y):=u_{\e_k}(x_k+\e_k X,y_k+\e_k Y), \quad (X,Y)\in B_{r_{\e_k}}(0),
\eeq
where $r_{\e_k}:=\f{\dist(P_{\e_k},\pa B_1)}{\e_k}\geq \f{1}{2\e_k}$.

The estimates
\beqo
\vert U_k\vert\leq M,\quad  \vert \na U_k\vert \leq M
\eeqo
give the following convergence (up to some subsequence)
\beqo
U_k\ri u \ \text{ uniformly on any compact set }K\subset \BR^2. 
\eeqo
$u$ solves the equation \eqref{main equation} and inherits the minimality from $U_k$. Finally by a diagonal argument we can get a further subsequence, which is still denoted by $\e_k$, such that for all $j\in\mathbb{N}^+$ the following limits hold true: 
\beqo
Q_j=\lim\limits_{k\ri \infty}Q^j_{\e_k},\ \  
R_j=\lim\limits_{k\ri \infty}R^j_{\e_k}.
\eeqo
It is easy to verify that $u$ and $\{Q_j, R_j\}_{j=1}^\infty$ satisfy Property (a), (b), (c) in Theorem \ref{main theorem}. The rest of this section is devoted to prove Property (d).  

Let
\beqo
u_t(z):=u(tz),\quad z\in \BR^2,\;t>0.
\eeqo
Note that this is the blow-down scaling for large $t$.

\textbf{Claim. } For any $t_k\ri\infty$, there is a subsequence (still denoted by $t_k$) such that  
\beq\label{L1 loc converge}
u_{t_k}\xrightarrow{L^1_{loc}(\BR^2)} u_0=\sum_{j=1}^{\bar{N}}\bar{a}_j\chi_{D_j},
\eeq
where $\bar{a}_j\in \{W=0\}=\{a_1,a_2,a_3\}$, $1\leq \bar{N}\leq 3$, $\mathcal{P}=\{D_j\}_{j=1}^{\bar{N}}$ is a minimal partition of $\BR^2$, $\pa\mathcal{P}$ is a minimal cone, i.e. $\pa\mathcal{P}$ is scaling invariant. 
\vspace{5pt}

\textbf{Proof of the Claim. }We follow the proof of Modica \cite[Theorem 3]{modica1979gamma} with the necessary modifications to fit our setting. First we establish the compactness in $L^1_{loc}(\BR^2)$. 

Since 
\begin{equation}\label{est: EtutBR}
\begin{split}
    E_t(u_t,B_r)&:=\int_{B_r}\left( \f{1}{2t}\vert \na u_t\vert ^2+tW(u_t) \right)\,dz \\
    &= \f{1}{t}\int_{B_{tr}}\left(  \f12\vert \na u\vert ^2+W(u) \right)\,dz\leq Cr,
    \end{split}
\end{equation}
where $C$ is independent of $r$, by a well-known estimate for minimizers (cf. \cite[Lemma 5.1]{afs-book}). By Baldo \cite[Proposition 2.1]{Baldo}, $\|u_t\|_{L^\infty(\BR^2)}<C$, and \eqref{est: EtutBR} we obtain compactness in $L^1(B_r)$, and further in $L^1_{loc}(\BR^2)$  by a diagonal argument. Utilizing \cite[Proposition 2.2 \& 2.4]{Baldo}, we conclude that for $u_0$ as in \eqref{L1 loc converge}, $\mathcal{P}$ is a minimal partition. To obtain that $\pa \mathcal{P}$ is a cone, following Modica, we repeat the argument by considering $u_0(tz)$ and the corresponding $\pa \mathcal{P}_t$. For this purpose we need the analog of Giusti \cite[Theorem 9.3]{giusti1984minimal} for minimal partitions, which are defined as flat chains of top dimension (see Fleming \cite{fleming1966flat} and White \cite{whitenotes}). The key points are 
\begin{itemize}
    \item[(1)] the flat norm coincides with the mass norm which in turn equals the $L^1$ norm,
    \item[(2)] the compactness theorem for flat chains,
    \item[(3)] the monotonicity formula,
\end{itemize}
which all can be checked (see e.g. the expository paper \cite{alikakos2013structure}).

Thus the claim is established. 

\vspace{5pt}

Since the only minimal cones of
dimension one in $\BR^2$ are the straight line and the triod, it suffices to show that $\pa \mathcal{P}$ is a triod instead of a straight line.

We argue by contradiction. Suppose $\pa \mathcal{P}$ is a straight line, then after a possible rotation we assume 
\beq
\begin{split}
&u_{t_k}\xrightarrow{L^1_{loc}(\BR^2)} u_0=\bar{a}_1\chi_{D_1}+\bar{a}_2\chi_{D_2},\\
&\{\bar{a}_1,\bar{a}_2\}\subset \{a_1,a_2,a_3\},\quad D_1=\{(x,y): y\geq 0\},\ D_2= \{(x,y): y<0\}. 
\end{split}
\eeq
By Caffarelli-C\'{o}rdoba density argument, the $L^1_{loc}$ convergence can be stengthened to uniform convergence away from $\{y=0\}$. Thus we have
\begin{align}
 \non   &\forall\, \e,\text{ there exists a }R=R(\e)\text{ such that }\\
\non    & \dist(u(z),\bar{a}_1)\leq \e,\quad \forall \ z\in B_R\cap \{y\geq \e R\},\\
\non    & \dist(u(z),\bar{a}_2)\leq \e,\quad \forall \ z\in B_R\cap \{y\leq  -\e R\},\\
\label{energy upper bound in the stripe}    &\int_{B_R} \left(\f12\vert \na u\vert^2+W(u)\right) \,dydx \leq R\sigma(2+\e).
\end{align}
Here the last estimate \eqref{energy upper bound in the stripe} follows from the estimate $\sigma \mathcal{H}^1(\pa\mathcal{P}\cap B_1)=\lim\limits_{k\ri\infty} E_{t_k}(u_{t_k}, B_1)$, thanks to the $\G$--convergence result in Baldo\cite{Baldo} that holds also without the mass constraint (see Gazoulis\cite{gazoulis}).

\begin{case}
When $a_1\in\{\bar{a}_1,\bar{a}_2\}$, without loss of generality we assume $\bar{a}_1=a_1$, $\bar{a}_2=a_2$. 
Take $\e$ to be a small constant whose value will be determined in the proof process. We will focus on the stripe $\{z=(x,y):\; x\in[-R,R],\,y\in[-\e R,\e R] \}$ (written as $[-R,R]\times [-\e R,\e R]$) where $R=R(\e)$. 

By Property (b) in Theorem \ref{main theorem} there are points $R_1,..., R_N \in [-R,R]\times [-\e R,\e R]$ ($N\geq \f{R}{32C_0}$) that satisfy
\begin{align*}
    &\vert u(R_i)-a_3\vert\leq \g,\quad \forall\ 1\leq i\leq N,\\
    &\dist(R_i,R_j)\geq 6C_0,\quad \forall\ 1\leq i,j\leq N,\\
    &\vert u(z)-a_3\vert\leq 2\g, \quad \forall\ z\in B(R_i, \f{\g}{M}).
\end{align*}

Define
\beqo
X=\bigcup\limits_{i=1}^N [x_i-\f{\g}{M},x_i+\f{\g}{M}].
\eeqo
Here $x_i$ is the abscissa of $R_i$.

If $\mathcal{H}^1(X)\geq \delta R$, for some small $\delta$ which will be determined later, then we can estimate
\begin{equation}\label{est: energy when X<delta R}
    \begin{split}
    E(u,B_R)&= \int_{-R\sqrt{1-\e^2}}^{R\sqrt{1-\e^2}}\int_{-\e R}^{\e R} \left(\f12\vert\na u\vert^2+W(u)\right) \,dydx\\
    &=\left(\int_X +\int_{[-R\sqrt{1-\e^2}, R\sqrt{1-\e^2}]\setminus X}\right) \,dx \left( \int_{-\e R}^{\e R} \left(\f12\vert\pa_y u\vert^2+W(u)\right) \,dy  \right)\\
    &\geq \delta R\cdot (2\sigma (1-C\g^2))+(2\sqrt{1-\e^2}R-\delta R)\sigma (1-C\e^2)\\
    &\geq R\sigma (2+\delta -2C\delta\g^2-2C\e^2)\\
    &\geq R\sigma (2+\f{\delta}{2}),
    \end{split}
\end{equation}
where to get last inequality we have required 
\beqo
2C\g^2\leq \f14,\quad 2C\e^2\leq \f{\delta}{4}.
\eeqo
Note that the first inequality in the estimate above is derived from the fact that when $x\in X$, on the vertical line $\{x\}\times [-\e R,\e R]$ $u$ is $\e$--close to $a_1, a_2$ at $(x,\e R),(x,-\e R)$ respectively, and is $\g$--close to $a_3$ at some middle point $(x,y)$. 

As a result, \eqref{est: energy when X<delta R} contradicts with \eqref{energy upper bound in the stripe} if we take $\e<\f{4}{\delta}$.

Otherwise $\mathcal{H}^1(X)<\delta R$. We define
\beqo
\tilde{X}=\bigcup\limits_{i=1}^N [x_i-2C_0,x_i+2C_0].
\eeqo
Note that we can assume $\g$ satisfies $\f{\g}{M}\leq C_0$. Then we have 
\beqo
\mathcal{H}^1(\tilde{X})< \f{2C_0M}{\g}\cdot \delta R.
\eeqo

For every $R_i$, by Property (c) in Theorem \ref{main theorem} there exists a point $P_i\in B(R_i,C_0)$ such that $\vert u(P_i)-a_1\vert\leq \g$. It follows that for any $z\in B(P_i,\f{\g}{M})$, $\vert u(z)-a_1\vert \leq 2\g$. This allows us to estimate the energy from below inside $B(R_i,2C_0)$. Actually after suitable translation and rotation we can assume 
\beqo
R_i=(0,0),\quad P_i=(x_i,0),\quad\text{for some }x_i\in(0,C_0].  
\eeqo
For any $y_0\in[-\f{\g}{M},\f{\g}{M}]$, the line segment $\{(x,y_0): 0\leq x\leq x_i\}$ satisfies 
\beqo
\vert u((0,y_0))-a_3\vert\leq 2\g,\quad \vert u((x_i,y_0))-a_1\vert\leq 2\g.
\eeqo
Thus by calculating the energy along these line segments that ``connect" $B(P_i,\f{\g}{m})$ and $B(R_i,\f{\g}{M})$, one has
\beq
\begin{split}
&\int_{B(R_i,2C_0)}\left( \f12\vert\na u\vert^2+W(u)  \right)\,dz\\
\geq & \int_{-\f{\g}{M}}^{\f{\g}{M}}\int_{0}^{x_i}\left( \f12\vert\pa_x u\vert ^2+W(u) \right)\,dxdy\\
\geq &\f{2\g}{M}(\sigma-C\g^2)\geq \f{\g}{M}\sigma.
\end{split}
\eeq

Note that since $\dist(R_i,R_j)\geq 6C_0$ for any $i\neq j$, all $B(R_i,2C_0)$ are pairwisely disjoint. We have
\begin{equation}\label{est energy when X>delta R}
    \begin{split}
       &\int_{B_R}\left( \f12\vert \na u\vert^2+W(u) \right)\,dz\\
       \geq & \int_{[-\sqrt{1-\e^2}R,\sqrt{1-\e^2}R]\setminus \tilde{X}}\int_{-\e R}^{\e R}\left(\f{1}{2}\vert\pa_y u\vert^2+W(u)  \right)\,dydx\\
       &\qquad\qquad+\sum\limits_{i=1}^N \int_{B(R_i,2C_0)}\left(\f12\vert\na u\vert^2+W(u)\right)\,dz\\
       \geq & (2-\e^2-\delta\cdot \f{2C_0M}{\g})(1-C\e^2)R\sigma +\f{\g}{M}\sigma\cdot N\\
       \geq & R\sigma\left[ (2-\e^2-\f{2C_0M}{\g}\delta)(1-C\e^2)+\f{\g}{32C_0M}  \right]\\
       \geq & R\sigma \left( 2+\f{\g}{32C_0M} -(2C+1)\e^2-\f{2C_0M}{\g}\delta \right)\\
       \geq &R\sigma\left(2+ \f{\g}{64C_0M} \right),
    \end{split}
\end{equation}
where in the last inequality we reuire that $(2C+1)\e^2<\f{\g}{128C_0M}$ and $\delta<\f{\g^2}{256\vert C_0M\vert ^2}$. If we further require that $\e<\f{\g}{128C_0M}$, then \eqref{est energy when X>delta R} contradicts with the upper bound \eqref{energy upper bound in the stripe}.

\end{case}

\begin{case}
If $\{a_2,a_3\}=\{\bar{a}_1,\bar{a}_2\}$, then for any $r\in(0,R)$, there exists $P(r)$ such that $\vert u(P(r))-a_1\vert\leq \g$. And $P(r)\in [-R,R]\times [-\e R,\e R]$. Let 
\beqo
r_i:=6iC_0,\quad P_i:=P(r_i),\quad i=1,2,...,N,\ N=[\f{R}{6C_0}].
\eeqo
Then $\{P_i\}_{i=1}^N$ is a sequence of points in the stripe $ [-R,R]\times [-\e R,\e R]$ and satisfy
\begin{align*}
    &\dist(P_i,P_j)\geq 6C_0,\\
  \forall\ & P_i,\ \exists\; z_i\in \overline{B}(P_i,C_0) \text{ such that }\min\{\vert u(z_i)-a_2\vert,\vert u(z_i)-a_3\vert\}\leq \g.
\end{align*}

Then one can replace $\{R_i\}$ by $\{P_i\}$ in the proof of Case 1 and get a contradiction using the same arguments.
\end{case}

Therefore, we conclude that $\bar{N}=3$ and that $\pa\mathcal{P}$ is a triod. 

Finally \eqref{convergence in main thm along a rays} follows from \cite[Proposition 5.6]{afs-book}, and from $\pa\mathcal{P}$ being the triod. The proof of Theorem \ref{main theorem} is complete.
\end{proof}

\appendix
\section{Proof of Lemma \ref{lemma:upper bound}}\label{app:upper bound}
For sufficiently small $\e$, we want to construct an energy competitor $u_{test}\in W^{1,2}(B_1,\BR^2)$ with the boundary condition $u_{test}\big|_{\pa B_1}=g_\e$ and show that 
\beq\label{energy for utest}
\int_{B_1}\left( \f{\e}{2}|\na u_{test}|^2+\f{1}{\e}W(u_{test}) \right)\,dz\leq 3\sigma+C\e,
\eeq
for a positive constant $C$ independent of $\e$. 

Let $c_0$ be as in \eqref{def of g_eps}. Set
\beqo
r_1:=c_0\e,\quad r_2:=1-c_0\e.
\eeqo

We define 
\beqo
\begin{split}
S(s,t;\theta_1,\theta_2)&:=\{(x,y)=(r\cos\theta,r\sin\theta):s\leq r\leq t,\theta_1\leq \theta\leq \theta_2 \},\\ &0\leq s<t\leq 1,\ \ 0\leq \theta_1<\theta_2\leq 2\pi.
\end{split}
\eeqo

Since $W$ satisfies (H2), we only need to construct $u_{test}$ and estimate its energy in the sector $S(0,1;\f{\pi}{3},\pi)$. Then one can do the same construction and estimate on the other two sectors $S(0,1;\pi,\f53\pi)$ and $S(0,1;0,\f{\pi}{3})\cup S(0,1;\f{5\pi}3,2\pi)$.

Recall that $U_{12}\in W_{loc}^{1,2}(\BR,\BR^2)$ is the 1D minimizer of the minimization problem 
\beqo
\min \int_{-\infty}^\infty \left( \f12|v'|^2+W(v) \right)\,d\eta,\quad  \lim\limits_{x\ri-\infty} v(\eta)=a_1,\ \lim\limits_{\eta\ri\infty} v(\eta)=a_2, \ v(\BR)\subset \BR^2\setminus \{a_1,a_2\}.
\eeqo
The properties of this 1D minimizer play an important role in our construction of $u_{test}$. 

Now we are ready to construct $u_{test}$. On $S(r_1,r_2; \f{\pi}{3},\f{2\pi}{3})$, we set 
\beqo
u_{test}=U_{12}(\f{r\sin(\f{\pi}{2}-\theta)}{\e})
\eeqo

By the exponential decay estimate for the minimizing connection $U_{12}$ (see \cite[Proposition 2.4]{afs-book}), we have the following estimates for $u_{test}$ on $\pa S(r_1,r_2;\f{\pi}{3},\f{2\pi}{3})$:
\begin{align}
\label{bdy data pi/3} &|u_{test}(r,\f{\pi}{3})-a_2|\leq Ke^{-k\f{r}{2\e}},  \ \text{ on }\{r_1\leq r\leq r_2, \theta=\f{\pi}{3}\}\\
\label{bdy data 2pi/3} &|u_{test}(r,\f{2\pi}{3})-a_1|\leq Ke^{-k\f{r}{2\e}},  \ \text{ on }\{r_1\leq r\leq r_2, \theta=\f{2\pi}{3}\}\\
\label{bdy data r1 1} & |u_{test}(r_1,\theta)-a_1|\leq Ke^{-kc_0\sin (\theta-\f{\pi}{2})}, \ \text{ on }\{r=r_1, \f{\pi}{2}\leq  \theta\leq \f{2\pi}{3}\},\\
\label{bdy data r1 2} & |u_{test}(r_1,\theta)-a_2|\leq Ke^{-kc_0\sin (\f{\pi}{2}-\theta)}, \ \text{ on }\{r=r_1, \f{\pi}{3}\leq  \theta\leq  \f{\pi}{2}\},\\
\label{bdy data r2 1} & |u_{test}(r_2,\theta)-a_1|\leq Ke^{-k\f{r_2\sin (\theta-\f{\pi}{2})}{\e}}, \ \text{ on }\{r=r_2, \f{\pi}{2}\leq  \theta\leq \f{2\pi}{3}\},\\
\label{bdy data r2 2} & |u_{test}(r_2,\theta)-a_2|\leq Ke^{-k\f{r_2\sin (\f{\pi}{2}-\theta)}{\e}}, \ \text{ on }\{r=r_2, \f{\pi}{3}\leq  \theta\leq  \f{\pi}{2}\}.
\end{align}
Here $K,k$ are constants that are independent of $\e$.

Note that in $S(r_1,r_2;\f{\pi}{3},\f{2\pi}{3})$, $u_{test}$ only depends on the $x$ variable, thus $|\pa_y u_{test}|=0$. We compute the energy
\begin{equation}\label{energy in r1 r2 pi/3 2pi/3}
\begin{split}
    &\int_{S(r_1,r_2,\f{\pi}{3},\f{2\pi}{3})} \left(  \f{\e}{2}|\na u_{test}|^2+\f{1}{\e}W(u_{test}) \right)\,dz\\
    =&\int_{S(r_1,r_2,\f{\pi}{3},\f{2\pi}{3})} \left(  \f{\e}{2}|\pa_x u_{test}|^2+\f{1}{\e}W(u_{test}) \right)\,dz\\
    \leq& \int_{r_1}^{r_2}\,dy \int_{-\f{\sqrt{3}}{3}y}^{\f{\sqrt{3}}{3}y}\left(    \f{\e}{2}|\pa_x u_{test}|^2+\f{1}{\e}W(u_{test}) \right)\,dx\\
    \leq &(r_2-r_1)\sigma\\
    \leq &(1-c_0\e)\sigma.
\end{split}
\end{equation}

Similarly we set respectively
\beqo
\begin{split}
&u_{test}=U_{31}(\f{r\sin{(\f76\pi-\theta)}}{\e}), \quad \text{on }S(r_1,r_2;\pi,\f{4\pi}{3}),\\
&u_{test}=U_{23}(\f{r\sin{(\f{11}6\pi-\theta)}}{\e}), \quad \text{on }S(r_1,r_2;\f53 \pi,2\pi).
\end{split}
\eeqo
In particular, when $\t=\pi$ it holds that
\beq
\label{bdy data pi} |u_{test}(r,\pi)-a_1|\leq Ke^{-k\f{r}{2\e}},\ \text{ on }\{r_1\leq r\leq r_2,\,\theta=\pi\}.
\eeq
With \eqref{bdy data 2pi/3} and \eqref{bdy data pi} we are able to define $u_{test}$ on $S(r_1,r_2;\f{2\pi}{3},\pi)$ by
\beqo
u_{test}(r,\theta):= u_{test}(r,\f{2\pi}{3})\f{\pi-\theta}{\f13\pi}+u_{test}(r,\pi)\f{\theta-\f23\pi}{ \f13\pi}.
\eeqo
And this construction can be extended to $S(r_1,r_2;\f{4\pi}{3},\f{5\pi}{3})$ and $S(r_1,r_2;0,\f{\pi}{3})$ in the same way. The energy in $S(r_1,r_2;\f{2\pi}{3},\pi)$ will be estimated in polar coordinates. 

\beq\label{energy in r1 r2 2pi/3 pi}
\begin{split}
    &\int_{S(r_1,r_2;\f{2\pi}{3},\pi)} \left( \f{\e}{2}|\na u_{test}|^2+\f{1}{\e}W(u_{test})  \right)\,dz\\
    =& \int_{\f{2\pi}{3}}^{\pi}\int_{r_1}^{r_2} \left( \f{\e}{2}(|\f{\pa u_{test}}{\pa r}|^2+\f{1}{r^2}|\f{\pa u_{test}}{\pa \theta}|^2)+\f{1}{\e}W(u_{test}) \right)r\,dr\,d\theta.
\end{split}
\eeq

Note that 
\begin{align*}
    &|\f{\pa u_{test}}{\pa r}|^2\leq 2(|\f{\pa u_{test}(r,\f{2\pi}{3})}{\pa r}|^2+|\f{\pa u_{test}(r,\pi)}{\pa r}|^2)\leq \f{1}{2\e^2} (|U_{12}'(-\f{r}{2\e})|^2+|U_{12}'(\f{r}{2\e})|^2)\leq \f{Ce^{-k\f{r}{\e}}}{\e^2},\\
    &\qquad\quad  |\f{\pa u_{test}}{\pa \theta}|^2\leq Ce^{-k\f{r}{\e}},\quad  W(u_{test}(r,\theta))\leq Ce^{-k\f{r}{\e}},
\end{align*}
where $C$ is a universal constant.

Substituting these into \eqref{energy in r1 r2 2pi/3 pi} yields
\beq\label{energy in r1 r2 2pi/3 pi final}
\begin{split}
    &\int_{S(r_1,r_2;\f{2\pi}{3},\pi)} \left( \f{\e}{2}|\na u_{test}|^2+\f{1}{\e}W(u_{test})  \right)\,dz\\
    \leq & \int_{\f{2\pi}{3}}^{\pi} \int_{c_0\e}^{1-c_0\e} \left(  \f{\e}{2}\left(  \f{Ce^{-k\f{r}{\e}}}{\e^2}+\f{Ce^{-k\f{r}{\e}}}{r^2} \right)+\f{1}{\e}Ce^{-k\f{r}{\e}} \right)r\,dr\,d\theta\\
    \leq& C(W,c_0)\e,
\end{split}
\eeq
for some constant $C(W,c_0)$ that does not depend on $\e$.

Now that $u_{test}$ has been already defined on the annulus $S(r_1,r_2;0,2\pi)$, we proceed to define $u_{test}$ in the inner ball $B_{r_1}$ and the outer layer $S(r_2,1; 0,2\pi)$. First of all we take $u_{test}$ to be the harmonic extension in $B_{r_1}$, with respect to its boundary data.
\beqo
\begin{split}
\Delta& u_{test}=0 \text{ in }B_{r_1},\\
u_{test}\big|_{\pa B_{r_1}} & \text{ is given by the construction on }S(r_1,r_2;0,2\pi).
\end{split}
\eeqo

It is not hard to verify that 
$$
|u_{test}|\leq C,\ \  |\nabla_T u_{test}|\leq \f{C}{\e} \ \text{ on }\pa B_{r_1},
$$
where $\na_T$ denotes the tangential derivative, $C=C(W,c_0)$ is a constant independent of $\e$. By elliptic regularity, $|u_{test}|$ and $\e|\na u_{test}|$ are also bounded by some universal constant $C$ inside $B_{r_1}$. Then we have 
\beq\label{energy in 0 r1}
\int_{B_{r_1}} \left( \f{\e}{2}|\na u_{test}|^2+\f{1}{\e}W(u_{test}) \right)\,dz
\leq \pi (c_0\e)^2\left(\f{\e}{2} (\f{C}{\e})^2+\f{1}{\e}C\right)\,dx\leq C(W,c_0)\e. 
\eeq

It remains to construct $u_{test}$ on the annulus $S(r_2,1;0,2\pi)$. Set
\beq
u_{test}(r,\theta)=\f{1-r}{c_0\e} u_{test}(r_2,\theta)+\f{r-r_2}{c_0\e}g_{\e}(\theta), \quad r_2\leq r\leq 1, \ \theta\in [0,2\pi).
\eeq
Here $g_\e$ is the boundary data on $\pa B_1$ defined by \eqref{def of g_eps} and $u_{test}(r_2,\theta)$ is given by the construction of $u_{test}$ on $\pa S(r_1,r_2;0,2\pi)$. We have
\begin{equation}\label{energy in r2 1 total}
    \begin{split}
     &\int_{S(r_2,1;\f{\pi}{3},\pi)} \left( \f{\e}{2}|\na u_{test}|^2+\f{1}{\e}W(u_{test})  \right)\,dz\\
    =&\left( \int_{\frac{\pi}{2}-c_0\e}^{\frac{\pi}{2}+c_0\e}+\int_{\f{\pi}{3}}^{\frac{\pi}{2}-c_0\e}+\int_{\frac{\pi}{2}+c_0\e}^{\f{2\pi}{3}}+\int_{\f{2\pi}{3}}^{\pi} \right) \int_{r_2}^{1} \left( \f{\e}{2}(|\f{\pa u_{test}}{\pa r}|^2+\f{1}{r^2}|\f{\pa u_{test}}{\pa \theta}|^2)+\f{1}{\e}W(u_{test}) \right)r\,dr\,d\theta.
    \end{split}
\end{equation}
We estimate each part separately. In $S(r_2,1;\frac{\pi}{2}-c_0\e,\frac{\pi}{2}+c_0\e)$, it holds that
\begin{align*}
     &\int_{\frac{\pi}{2}-c_0\e}^{\frac{\pi}{2}+c_0\e}\int_{r_2}^1 \left( \f{\e}{2}(|\f{\pa u_{test}}{\pa r}|^2+\f{1}{r^2}|\f{\pa u_{test}}{\pa \theta}|^2)+\f{1}{\e}W(u_{test}) \right)r\,dr\,d\theta.\\
     \leq& 2|c_0\e|^2 \cdot \left( \f{\e}{2}\bigg|\f{C}{\e}\bigg|^2+\f{1}{\e} C \right)\leq C\e;
\end{align*}

In $S(r_2,1; \f{\pi}{3},\frac{\pi}{2}-c_0\e)\cup S(r_2,1; \frac{\pi}{2}+c_0\e,\f{2\pi}{3})$, by \eqref{bdy data r2 1} and \eqref{bdy data r2 2} we have
\begin{align*}
     &\left(\int_{\f{\pi}{3}}^{\frac{\pi}{2}-c_0\e}+\int_{\frac{\pi}{2}+c_0\e}^{\f{2\pi}{3}}\right)\int_{r_2}^1 \left( \f{\e}{2}(|\f{\pa u_{test}}{\pa r}|^2+\f{1}{r^2}|\f{\pa u_{test}}{\pa \theta}|^2)+\f{1}{\e}W(u_{test}) \right)r\,dr\,d\theta.\\
     \leq& \left(\int_{\f{\pi}{3}}^{\frac{\pi}{2}-c_0\e}+\int_{\frac{\pi}{2}+c_0\e}^{\f{2\pi}{3}}\right)\int_{r_2}^1 \bigg(
     \f{\e}{2}\f{|u_{test}(r_2,\theta)-g_\e(\theta)|^2}{(c_0\e)^2}+\f{\e}{2r^2}(|\f{\pa u_{test}(r_2,\theta)}{\pa\theta}|^2+|\f{\pa g_\e(\theta)}{\pa \theta}|^2)\\
     &\qquad\qquad\qquad\qquad +\f{1}{\e} C|u_{test}(r_2,\theta)-g_\e(\theta)|^2\bigg)r\,dr\,d\theta\\
     \leq &C\int_{C_0\e}^{\f{\pi}{6}} \int_{r_2}^1 \left( \f{1}{\e} e^{-k\f{\sin\theta}{\e}}   \right)\,r\,dr\,d\theta \leq C\e;
\end{align*}

In $S(r_2,1;\f{2\pi}{3},\pi)$, we note that 
$$
u_{test}(r, \theta)-a_1= \f{r-r_2}{c_0\e} \left((u_{test}(r_2,\f{2\pi}{3})-a_1)\f{\pi-\theta}{\f13\pi}+(u_{test}(r_2,\pi)-a_1)\f{\theta-\f23\pi}{ \f13\pi}\right),
$$
which implies that 
\begin{align*}
     \int_{\f{2\pi}{3}}^{\pi}\int_{r_2}^1 \left( \f{\e}{2}(|\f{\pa u_{test}}{\pa r}|^2+\f{1}{r^2}|\f{\pa u_{test}}{\pa \theta}|^2)+\f{1}{\e}W(u_{test}) \right)r\,dr\,d\theta \leq Ce^{-\f{k}{\e}}\leq C\e.
\end{align*}

Therefore, \eqref{energy in r2 1 total} becomes
\beq\label{energy in r2 1 final}
\int_{S(r_2,1;\f{\pi}{3},\pi)} \left( \f{\e}{2}|\na u_{test}|^2+\f{1}{\e}W(u_{test})  \right)\,dz\leq C\e.
\eeq

Finally, using \eqref{energy in r1 r2 pi/3 2pi/3}, \eqref{energy in r1 r2 2pi/3 pi final}, \eqref{energy in 0 r1}, \eqref{energy in r2 1 final} we conclude that 
\beqo
\int_{S(0,1;\f{\pi}{3},\pi)}  \left( \f{\e}{2}|\na u_{test}|^2+\f{1}{\e}W(u_{test}) \right)\,dz\leq \sigma+ C(W,c_0)\e.
\eeqo
The energies on $S(0,1;\pi,\f{5\pi}{3})$ and $S(0,1;0,\f{\pi}{3})\cup S(0,1;\f{5\pi}3,2\pi)$ satisfy the same estimate. Adding them up leads to \eqref{upper bound}, which completes the proof.

\bibliographystyle{acm}
\bibliography{bib_triple_junction}

\end{document}